\documentclass[12pt]{amsart}
\usepackage{latexsym, amsmath, amssymb, amsthm}
\usepackage{bm, mathrsfs, euscript}
\usepackage{pstricks,pst-node,pst-text,pst-tree}
\usepackage[colorlinks=true,linkcolor=blue,citecolor=red]{hyperref}
\usepackage[shortalphabetic]{amsrefs}
\usepackage[centering, includeheadfoot, includehead, 
  hmargin=1.1in, vmargin=1in]{geometry}
\newcommand{\define}[1]{{\bfseries\itshape #1}}
\newtheorem{lemma}{Lemma}[section]
\newtheorem{theorem}[lemma]{Theorem}
\newtheorem{corollary}[lemma]{Corollary}
\newtheorem{proposition}[lemma]{Proposition}
\theoremstyle{definition}
\newtheorem{example}[lemma]{Example}
\newtheorem{remark}[lemma]{Remark}
\newtheorem*{acknowledgements}{Acknowledgements}
\renewcommand{\theequation}%
{\arabic{section}.\arabic{lemma}.\arabic{equation}}
\renewcommand{\thefigure}%
{\arabic{section}.\arabic{lemma}.\arabic{equation}}
\renewcommand{\geq}{\geqslant}
\renewcommand{\leq}{\leqslant}
\newcommand{\kk}{\ensuremath{\Bbbk}} 
\renewcommand{\AA}{\ensuremath{\mathbb{A}}} 
\newcommand{\NN}{\ensuremath{\mathbb{N}}} 
\newcommand{\PP}{\ensuremath{\mathbb{P}}} 
\newcommand{\ZZ}{\ensuremath{\mathbb{Z}}} 
\newcommand{\lex}{\ensuremath{>_{\!\! +}^{\,}}} 
\newcommand{\xel}{\ensuremath{>_{\!\! -}^{\,}}} 
\newcommand{\order}{\ensuremath{>_{\! I}^{\,}}} 
\DeclareMathOperator{\Hilb}{Hilb}
\DeclareMathOperator{\bHilb}{\textbf{Hilb}}
\DeclareMathOperator{\Hom}{Hom}

\DeclareMathOperator{\initial}{in}
\DeclareMathOperator{\codim}{codim}
\DeclareMathOperator{\rank}{rank}
\DeclareMathOperator{\Spec}{Spec}

\begin{document}

\title[Smooth and irreducible multigraded Hilbert schemes]{Smooth and
  Irreducible Multigraded Hilbert Schemes}

\author[D.~Maclagan]{Diane Maclagan}
\address{Mathematics Institute\\ Zeeman Building\\ University of
  Warwick\\ Coventry\\ CV4 7AL\\ United Kingdom}
\email{D.Maclagan@warwick.ac.uk}

\author[G.G.~Smith]{Gregory G. Smith} 
\address{Department of Mathematics \& Statistics \\ Queen's
  University \\ Kingston \\ ON \\ K7L 3N6\\ Canada}
\email{ggsmith@mast.queensu.ca}



\begin{abstract}
  The multigraded Hilbert scheme parametrizes all homogeneous ideals
  in a polynomial ring graded by an abelian group with a fixed Hilbert
  function.  We prove that any multigraded Hilbert scheme is smooth
  and irreducible when the polynomial ring is $\ZZ[x,y]$, which
  establishes a conjecture of Haiman and Sturmfels.
\end{abstract}

\maketitle

\section{Introduction}
\label{s:intro}

Hilbert schemes are the fundamental parameter spaces in algebraic
geometry.  Multigraded Hilbert schemes, introduced in
\cite{HaimanSturmfels}, consolidate many types of Hilbert schemes
including Hilbert schemes of points in affine space, toric Hilbert
schemes, $G$-Hilbert schemes for abelian $G$, and the original
Grothendieck Hilbert scheme.  The collection of all multigraded
Hilbert schemes contains many well-documented pathologies.  In
contrast, this paper identifies a surprisingly large subcollection of
multigraded Hilbert schemes that are both smooth and irreducible.

To be more explicit, let $S$ be a polynomial ring over $\ZZ$ that is
graded by an abelian group $A$.  A homogeneous ideal $I \subseteq S$
is admissible if, for all $a \in A$, the $\ZZ$-module $(S/I)_a =
S_a/I_a$ is a locally free with constant finite rank on $\Spec(\ZZ)$.
The Hilbert function $h_{S/I} \colon A \to \NN$ is defined by
$h_{S/I}(a) := \rank_\ZZ (S/I)_a$.  Given $h \colon A \to \NN$,
Theorem~1.1 of \cite{HaimanSturmfels} shows that there is a
quasiprojective scheme $\Hilb_S^h$ parametrizing all admissible
$S$-ideals with Hilbert function $h$.  Our main result is:

\begin{theorem}
  \label{t:main}
  If $S = \ZZ[x,y]$ is graded by an abelian group $A$, then for any
  function $h \colon A \to \NN$ the multigraded Hilbert scheme
  $\Hilb_S^h$ is smooth and irreducible.
\end{theorem}

\noindent
This theorem proves the conjecture in
\cite{HaimanSturmfels}*{Example~1.3} and
\cite{MillerSturmfels}*{Conjecture~18.46}.  Since $\Spec(\ZZ)$ is the
terminal object in the category of schemes, the theorem also extends,
via base change, to the category of $B$-schemes where $B$ is any
irreducible scheme.

The hypothesis that $S$ has two variables is essential in
Theorem~\ref{t:main}.  Example~1.4 of \cite{HaimanSturmfels}
demonstrates that multigraded Hilbert schemes can be reducible when
$S$ has three variables.  Even if one restricts to the standard
$\ZZ$-grading, Theorem~1.2 of \cite{CEVV} shows that irreducibility
fails; this also shows that the corank two result for toric Hilbert
schemes \cite{MaclaganThomas}*{Theorem~1.1} does not extend to all
multigraded Hilbert schemes.  Remarkably, especially when compared
with the connectedness of the Grothendieck Hilbert scheme
\cite{HartshorneThesis}*{Corollary~5.9}, Theorem~1 of \cite{Santos}
shows that multigraded Hilbert schemes can be disconnected.  This
evidence indicates that irreducibility of $\Hilb_S^h$ is rather
exceptional.

Similarly, one does not expect a general multigraded Hilbert scheme
$\Hilb_S^h$ to be smooth.  Indeed, the philosophy in
\cite{Vakil}*{\S1.2} suggests that most multigraded Hilbert schemes
contain complicated singularities.  For example, Theorem~1.1 of
\cite{Vakil} establishes that every singularity type of finite type
over $\Spec(\ZZ)$ appears on some $\Hilb_S^h$ when $S$ has at least
five variables.  With this in mind, Theorem~\ref{t:main} provides a
surprisingly comprehensive, but certainly not exhaustive, class of
smooth and irreducible multigraded Hilbert schemes.

We were particularly inspired by \cite{Evain}, although each basic
step in the proof of Theorem~\ref{t:main} has a counterpart in at
least one of the following papers: \cites{HartshorneThesis, Fogarty,
  Iarrobino, Reeves, Pardue, HaimanCatalan, Mall, Huibregtse,
MaclaganThomas, PeevaStillmanConnected, Fumasoli}.  The basic steps in
the proof are:
\begin{itemize}
\item[\textbf{(i)}] We prove that either $\Hilb_S^h \cong \PP^m \times
  \Hilb_S^{h'}$ or $\Hilb_S^h \cong \AA^m \times \Hilb_S^{h'}$ where
  $\Hilb_S^{h'}$ parametrizes ideals with codimension greater than
  one.
\item[\textbf{(ii)}] We identify a distinguished point on $\Hilb_S^h$
  and connect each point to this distinguished point by a rational
  curve.
\item[\textbf{(iii)}] We establish that the dimension of the tangent
  space is constant along these rational curves.
\item[\textbf{(iv)}] We show that the distinguished point on
  $\Hilb_S^h$ is nonsingular.
\end{itemize}
In all four steps, the combinatorial structure of the arguments allows
us to work over an arbitrary field $\kk$, so we are able lift our
results to multigraded Hilbert schemes over $\ZZ$.

\vspace*{1.2em}

The first step, which appears in \S2, shows that the multigraded
Hilbert scheme $\Hilb_S^h$ parametrizing codimension-one ideals
naturally splits into a product of a multigraded Hilbert scheme
parametrizing equidimensional ideals of codimension one and a
multigraded Hilbert scheme parametrizing ideals of higher codimension.
This is tantamount to proving that there exists a functorial
homogeneous factorization of the ideals with Hilbert function $h
\colon A \to \NN$.  Among the papers listed above, only
\cite{Fogarty}*{\S1} solves an analogous problem.  Nevertheless, our
factorization is striking because the primary decomposition of an
ideal needed not be homogeneous when the grading group $A$ has
torsion; see \cite{MillerSturmfels}*{Example~8.10}.  We establish this
decomposition when $S$ is a polynomial ring over $\kk$ with an
arbitrary number of variables.  In the two variable case it plays a
crucial role by reducing the proof of Theorem~\ref{t:main} to the
study of schemes $\Hilb_S^h$ parametrizing ideals with finite
colength.  As a result, the remaining three steps assume that $S =
\kk[x,y]$ and $h \colon A \to \NN$ has finite support.

In \S3, we distinguish a point on $\Hilb_S^h$ by imposing a partial
order on the set of all monomial ideals with Hilbert function $h
\colon A \to \NN$.  The distinguished point corresponds to the maximum
element in this poset, which we call the lex-most ideal.  In the
standard $\ZZ$-grading, the lex-most ideal coincides with the
lex-segment ideal and corresponds to the lexicographic point on the
Hilbert scheme.  The larger class of lex-most ideals is required
because lex-segment ideals do not necessarily exist for a general
$A$-grading; see Example~\ref{e:nonlex}.  In contrast with the
standard-graded case, a lex-most ideal may not have extremal Betti
numbers among all ideals with a given Hilbert function; see
Example~\ref{e:nonExtremalBetti}.  The uniqueness of the lex-most
ideal is the most novel aspect of the second step.

To complete the second step, we exhibit a chain of irreducible
rational curves connecting each point on $\Hilb_S^h$ to the
distinguished point.  Each curve comes from the Gr\"obner degenerations
of a binomial ideal.  The binomial ideals, which are edge ideals in
the sense of \cite{AltmannSturmfels}, arise from certain tangent
directions.  To designate a tangent direction, we use a combinatorial
model for the tangent space to $\Hilb_S^h$ at a point corresponding to
a monomial ideal.  Our model extends the ``cleft-couples'' in
\cite{Evain}*{\S2} and generalizes the ``arrows'' in
\cite{HaimanCatalan}*{\S2}.  Unlike \cite{Mall} and
\cite{PeevaStillmanConnected}, we cannot restrict to Borel-fixed
ideals because such ideals do not exist for arbitrary gradings.  This
approach has the advantage of proving that $\Hilb_S^h$ is rationally
chain connected.

The third step, found in \S4, identifies the tangent space to
$\Hilb_S^h$ at each point along these rational curves with a linear
subvariety of affine space.  Finding the dimension of the tangent
space is thereby equivalent to computing the rank of an explicit
system of linear equations.  Despite the conceptual simplicity, the
inevitable combinatorial analysis is rather intricate.  If we were
working over an algebraically closed field of characteristic zero,
then we could bypass this step by combining \cite{Iversen} and
\cite{Fogarty}*{Theorem~2.4}.  Dealing with an explicit system of
equations remarkably yields a higher level of generality.

For the fourth and final step, we demonstrate that the point on
$\Hilb_S^h$ corresponding the the lex-most ideal is nonsingular.  This
superficially resembles the smoothness of the lexicographic point in
the original Grothendieck Hilbert scheme; see
~\cite{ReevesStillman}*{Theorem~1.4}.  From the previous step we know
the dimension of the tangent space to $\Hilb_S^h$ at the distinguished
point.  To show that $\Hilb_S^h$ has the correct dimension at this
point, it suffices to embed an affine space of the correct dimension
into a neighborhood of the distinguished point.  Following
\cite{Evain}*{Proposition~10}, we achieve this in \S5 by building an
appropriate ideal that has the lex-most ideal as an initial ideal.
The last section of the paper also contains the proof of
Theorem~\ref{t:main}.

Earlier work on the geometry of multigraded Hilbert schemes
$\Hilb_S^h$ restricted either the possible grading groups $A$ or the
possible Hilbert functions $h \colon \NN \to A$.  In contrast,
Theorem~\ref{t:main} limits only the number of variables.  Indeed, our
set-up deliberately includes gradings, called
nonpositive~\cite{MillerSturmfels}*{Definition 8.7}, of $S$ for which
the grading group $A$ has torsion or $\rank_\ZZ S_a = \infty$ for some
$a \in A$.  Unsurprisingly, the nonpositive gradings are the primary
source of technical challenges.  In fact, all four steps would be
substantially easier if one excluded these cases.

Our success within this general framework leads to new questions: Can
one characterize a larger collection of connected multigraded Hilbert
schemes?  When the polynomial ring $S$ has more than two variables,
does there exist a unique lex-most ideal?  Do the maximal elements in
the poset of monomial ideals with a given Hilbert function correspond
to a nonsingular points?

\begin{acknowledgements}
  We thank Mark Haiman, Mike Roth, Jason Starr, Bernd Sturmfels and
  Mauricio Velasco for useful conversations.  The computer software
  \emph{Macaulay~2}~\cite{M2} was indispensable for generating
  examples.  The first author was partially supported by NSF grant
  DMS-0500386 and the Warwick North American Travel Fund; the second
  author was partially supported by NSERC.
\end{acknowledgements}

\subsection*{Conventions}
Throughout the paper, $\kk$ is a field and $\NN$ is the set of
nonnegative integers.  We write $\delta_{i,j}$ for the Kronecker
delta: $\delta_{i,j} = 1$ if $i = j$ and $0$ otherwise.  The
lexicographic order on $\kk[x,y]$ with $x > y$ is denoted by $\lex$
and the lexicographic order on $\kk[x,y]$ with $x < y$ is denoted by
$\xel$.  For an ideal $I \subseteq \kk[x,y]$, $\initial_{\lex}(I)$ and
$\initial_{\xel}(I)$ are the initial ideals of $I$ with respect to
$\lex$ and $\xel$.

\section{Factoring Multigraded Hilbert schemes}
\label{s:factoring}

We show in this section that the scheme $\Hilb_S^h$ naturally splits
into a product of a multigraded Hilbert scheme parametrizing
equidimensional codimension-one ideals and a multigraded Hilbert
scheme parametrizing ideals with codimension greater than one.  Let
$\kk$ be a field, let $A$ be an abelian group, and let $S :=
\kk[\bm{x}] = \kk[x_1, \dotsc, x_N]$ be an $A$-graded polynomial ring
with $N \geq 2$.  Unlike the other sections, we do not assume that $N
= 2$ in this section of the paper.  We begin with a description of the
multigraded Hilbert schemes parametrizing principal ideals.

\begin{lemma}
  \label{l:principal}
  Let $f \in S$ be a homogeneous polynomial of degree $d \in A$ such
  that the ideal $I := \langle f \rangle$ is admissible.  If $h \colon
  A \rightarrow \NN$ is the Hilbert function of $S/I$ and $m := h(d)$,
  then
  \[
  \Hilb_S^h \cong 
  \begin{cases}
    \PP^m & \text{if $\dim_{\kk} S_0 < \infty$,} \\
    \AA^m & \text{if $\dim_{\kk} S_0 = \infty$.}
  \end{cases}
  \]
\end{lemma}

\begin{proof}
  To begin, assume that $\dim_{\kk} S_0 < \infty$.  By
  \cite{MillerSturmfels}*{Theorem~8.6}, we have $\dim_{\kk} S_a <
  \infty$ for all $a \in A$, and $S_0 = \kk$.  Thus the Hilbert
  function $h_S \colon A \rightarrow \mathbb N$ given by $h_S (a) =
  \dim_{\kk} S_a$ is well-defined.  Multiplication by $f$ produces the
  short exact sequence
  \[
  0 \to S(-d) \to S \to S/I \to 0 \, ,
  \] 
  which shows that $h(a) = h_S(a) - h_S(a-d)$.  Since $S_0 = \kk$, it
  follows that $h(d) = h_S(d)-1$, so $\dim_\kk(J_d) = 1$ for any ideal
  $J$ with Hilbert function $h \colon A \to \NN$.  Applying this
  analysis to an element $g \in J_d$, we conclude that $J = \langle g
  \rangle$, so all ideals with Hilbert function $h \colon A \to \NN$
  are principal and generated in degree $d$.  Hence, $\Hilb_S^h$
  parametrizes the one-dimensional subspaces of $S_d$; in the language
  of \cite{HaimanSturmfels}*{\S 3}, the set $\{d\}$ is very
  supportive.  Therefore, we have $\Hilb_S^h \cong \PP^m$.

  Secondly, assume that $\dim_{\kk} S_0 = \infty$.  The hypothesis
  that $I$ is admissible places significant restrictions on $S_0$.
  Let $\bm{x}^{\bm{u}}$ be the initial term of $f$ with respect to
  some monomial order on $S$.  By \cite{Eisenbud}*{Theorem~15.3}, the
  monomials not divisible by $\bm{x}^{\bm{u}}$ form a $\kk$-basis for
  $S/I$.  Since $I = \langle f \rangle$ is admissible, all but
  finitely many monomials in $S_0$ are divisible by $\bm{x}^{\bm{u}}$.
  It follows that $S_0$ has a homogeneous system of parameters
  consisting of a single element, so the Krull dimension of $S_0$ is
  $1$.  The ring $S_0$ is a normal semigroup ring by
  \cite{MillerSturmfels}*{page 150}, so we deduce that $S_0 =
  \kk[\bm{x}^{\bm{v}}]$ for some monomial $\bm{x}^{\bm{v}} \in S$.

  We next examine the $S$-module structure of the graded component
  $S_d$.  Let $r \in \NN$ be the largest nonnegative integer with
  $\bm{u} - r \bm{v} \in \NN^N$ and set $\bm{w} := \bm{u} - r \bm{v}$.
  Let $\bm{x}^{\bm{w}'}$ be another monomial of degree $d$.  Since
  $\dim_{\kk} (S/I)_{d} < \infty$, all but finitely many monomials in
  $S_d$ are divisible by $\bm{x}^{\bm{u}}$.  Hence, the monomial
  $\bm{x}^{\bm{w}' + s \bm{v}}$ is divisible by $\bm{x}^{\bm{u}}$ for
  all $s \gg 0$, so for such $s$ we have $\bm{w}'' := (\bm{w}' + s
  \bm{v}) - (\bm{w} + r \bm{v}) = (\bm{w}'-\bm{w})+ (s-r)\bm{v} \in
  \NN^N$ with $\deg(\bm{x}^{\bm{w}''}) = 0$.  Thus $\bm{w}''=\ell
  \bm{v}$ for some $\ell \in \NN$ and so $\bm{w}' - \bm{w}$ is a
  multiple of $\bm{v}$.  By the construction of $\bm{w}$ this multiple
  must be nonnegative, so $S_d$ has $\kk$-basis $\{\bm{x}^{\bm{w} + s
    \bm{v}} : s \in \NN \}$.

  Finally, any monomial ideal with Hilbert function $h$ must contain
  $\bm{x}^{\bm{w} + m \bm{v}}$.  Since $\langle \bm{x}^{\bm{u}}
  \rangle$ has Hilbert function $h$, we have $r = m$ and there is only
  one monomial ideal with this Hilbert function; in the language of
  \cite{HaimanSturmfels}*{\S 3} the set $\{d\}$ is very supportive.
  Therefore, $\Hilb_S^h$ parametrizes the ideals of the form $\langle
  \bm{x}^{\bm{u}} + c_1 \bm{x}^{\bm{u} - \bm{v}} + \dotsb + c_m
  \bm{x}^{\bm{u} - m \bm{v}} \rangle$ where $c_j \in \kk$, so
  $\Hilb_S^h \cong \AA^m$.
\end{proof}

The next lemma contains the necessary algebraic preliminaries for
factoring multigraded Hilbert schemes.  The proof is complicated by
our need to work over polynomial rings with coefficients in an
arbitrary Noetherian $\kk$-algebra.

\begin{lemma}
  \label{l:alg}
  Let $K$ be a Noetherian $\kk$-algebra, let $R := K \otimes_\kk S$ be
  the $A$-graded polynomial ring with coefficients in $K$, and let $I$
  be a $R$-ideal.  If $J$ is the intersection of the codimension-one
  primary components of $I$ and $Q := (I:J)$, then $J$ and $Q$ are
  homogeneous, $J$ is a locally principal $K$-module, and $I = JQ$.
  Moreover, if $\Spec(K)$ is connected and $I$ is admissible, then
  both $J$ and $Q$ are admissible ideals.
\end{lemma}

\begin{remark}
  \label{r:gcd}
  The empty intersection of ideals equals $R$ by convention, and $N
  \geq 2$, so $J \neq 0$.  If $K$ is a unique factorization domain,
  then $J$ is simply generated by a greatest common divisor of any
  generating set for $I$.  This follows from observation that in a
  unique factorization domain a primary ideal whose radical has
  codimension one is principal.
\end{remark}

\begin{proof}[Proof of Lemma~\ref{l:alg}]
  We first show that $J$ is homogeneous.  Since we may assume that
  $\deg \colon \NN^N \to A$ is surjective, the structure theorem for
  finitely generated abelian groups implies that $A \cong \ZZ^r \oplus
  \ZZ/m_1 \ZZ \oplus \dotsb \oplus \ZZ/m_s \ZZ$.  It suffices to show
  that $J$ is homogeneous with respect to each summand of $A$.  The
  codimension-one primary components of $I$ are homogeneous with
  respect a torsion-free grading by \cite{Bourbaki}*{IV \S3.3
    Proposition 5}, so $J$ is also homogeneous with respect to a
  torsion-free grading.  The case $A = \ZZ/m\ZZ$ remains.  Consider an
  integral extension $\kk'$ of the field $\kk$ containing an
  $m^{\text{th}}$ root of unity $\omega$, and set $R' := \kk'
  \otimes_\kk R$.  Let $I' := \kk' \otimes_\kk I$ and let $J'$ be the
  codimension-one equidimensional component of $I'$.  From the
  intrinsic descriptions $J = \{ f \in R : \codim (I : f) \geq 2 \}$
  and $J' = \{f \in R' : \codim (I' : f) \geq 2 \}$, we see that $J =
  R \cap J'$.  Thus, it is enough to show that $J'$ is homogeneous
  with respect to a $(\ZZ/m\ZZ)$-grading.

  To accomplish this, fix generators for $J'$.  For a generator $f \in
  R'$, we write $f = \sum_{a \in A} f_a$ where each $f_a$ is
  homogeneous of degree $a \in A$.  We may assume that the generating
  set for $J'$ has been chosen so that $f_a$ does not lie in $J'$ if
  $f \neq f_a$ and $f_a \neq 0$.  Consider the automorphism $\phi
  \colon R' \to R'$ defined by $\phi(x_i) = \omega^{\deg(x_i)} x_i$
  for $1 \leq i \leq N$.  Since $\phi$ permutes the set of
  codimension-one primary components of $I'$, we have $\phi(J') = J'$.
  If $f_a \neq 0$, then $\omega^a f - \phi(f) = \sum_{a' \in A}
  (\omega^a - \omega^{a'}) f_{a'} \in J'$ has fewer homogeneous parts.
  Iterating this procedure, it follows that one of the nonzero $f_a$
  lies in $J'$ which means that $f$ is itself homogeneous.  Therefore,
  $J'$ has a homogeneous set of generators and $J$ is homogeneous.

  Next, consider $\mathfrak{p} \in \Spec(K)$ and let $k(\mathfrak{p})
  := K_{\mathfrak{p}}/\mathfrak{p} K_{\mathfrak{p}}$ be the residue
  field at $\mathfrak{p}$.  It follows from Remark~\ref{r:gcd} that $J
  \otimes_K k(\mathfrak{p})$ is generated by the greatest common
  divisor of a generating set for $I \otimes_K k(\mathfrak{p})$.
  Since the ideal $\mathfrak{p} R_{\mathfrak{p}}$ lies in the Jacobson
  radical of $R_{\mathfrak{p}} := R \otimes_K K_{\mathfrak{p}}$,
  Nakayama's Lemma implies that $J_{\mathfrak{p}} := J \otimes_K
  K_{\mathfrak{p}}$ is generated by a single element $f$, so the ideal
  $J$ is a locally principal $K$-module and $f \not\in \mathfrak{p}
  R_{\mathfrak{p}}$.

  To complete the first part, we examine $Q := (I:J)$.  Since $I$ and
  $J$ are homogeneous, the ideal $Q$ is as well.  To see that $I = J
  Q$, it suffices to regard these ideals as $K$-modules and work
  locally.  Suppose that $\mathfrak{p} \in \Spec(K)$ and
  $I_{\mathfrak{p}} := I \otimes_K K_{\mathfrak{p}} = \langle f_1,
  \dotsc, f_\ell \rangle$.  Since $I_{\mathfrak{p}} \subseteq
  J_{\mathfrak{p}} = \langle f \rangle$, we must have $f_i = f f'_i$
  for some $f'_i \in R_{\mathfrak{p}}$.  If $g \in Q_{\mathfrak{p}} =
  (I_{\mathfrak{p}}:J_{\mathfrak{p}})$ then $fg = \sum g_i f_i$ for
  some $g_i \in R_{\mathfrak{p}}$, so $f(g-\sum g_i f'_i) = 0$.
  Because $f$ either generates a codimension-one ideal or is a unit,
  it is not a zerodivisor, so $h \in \langle f'_1, \dots, f'_\ell
  \rangle$.  We conclude that $Q_{\mathfrak{p}}= \langle
  f'_1,\dots,f'_\ell \rangle$ and $I_{\mathfrak{p}} = J_{\mathfrak{p}}
  Q_{\mathfrak{p}}$.
 
  It remains to show that $J$ and $Q$ are admissible ideals.  Let $d
  := \deg(f) \in A$. Since the homogeneous generator $f$ of
  $J_{\mathfrak{p}}$ is not zerodivisor, there is a short exact
  sequence of $K_{\mathfrak{p}}$-modules $0 \to
  (R_{\mathfrak{p}})_{a-d} \to (R_{\mathfrak{p}})_a \to
  (R_{\mathfrak{p}}/ J_{\mathfrak{p}})_a \to 0$ for each $a \in A$.
  Since $f \not \in \mathfrak{p}R_{\mathfrak{p}}$, this sequence shows
  that $\operatorname{Tor}_{R_{\mathfrak{p}}}^1\bigl( k(\mathfrak{p})
  , (R_{\mathfrak{p}}/J_{\mathfrak{p}})_a \bigr) = 0$.  The surjection
  $(R_{\mathfrak{p}}/I_{\mathfrak{p}})_a \rightarrow
  (R_{\mathfrak{p}}/J_{\mathfrak{p}})_a$ of $K_{\mathfrak{p}}$-modules
  establishes that $(R_{\mathfrak{p}}/J_{\mathfrak{p}})_a$ is finitely
  presented.  Hence, Corollary~2 to \cite{Bourbaki}*{II \S 3.2
    Proposition 5} implies that
  $(R_{\mathfrak{p}}/J_{\mathfrak{p}})_a$ is free as a
  $K_{\mathfrak{p}}$-module for all $a \in A$.  Multiplication by $f$
  also produces the short exact sequence
  \begin{equation}
    \label{eq:add}
    0 \to (R_{\mathfrak{p}}/Q_{\mathfrak{p}})_{a-d}
    \to  (R_{\mathfrak{p}}/I_{\mathfrak{p}})_a \to
    (R_{\mathfrak{p}}/J_{\mathfrak{p}})_a \to 0 \,.
  \end{equation}  
  The admissibility of $I$ guarantees that
  $(R_{\mathfrak{p}}/I_{\mathfrak{p}})_a$ is a finite rank free
  $K_{\mathfrak{p}}$-module for all $a \in A$.  The sequence
  \eqref{eq:add} splits, so $(R_{\mathfrak{p}}/Q_{\mathfrak{p}})_{a}$
  is free $K_{\mathfrak{p}}$-module of finite rank and
  $(R_{\mathfrak{p}}/J_{\mathfrak{p}})_{a}$ has finite rank.  Since
  rank is upper semicontinuous, $(R/I)_a$ has constant rank on
  $\Spec(K)$, and $\Spec(K)$ is connected, we conclude that $(R/Q)_a$
  and $(R/J)_a$ have constant rank on $\Spec(K)$ for all $a \in A$.
\end{proof}

Before factoring multigraded Hilbert schemes, we record a geometric
observation.

\begin{lemma} 
  \label{l:codim}
  Given a function $h \colon A \to \NN$, there is a constant $c =
  c(h)$ such that, for each Noetherian $\kk$-algebra $K$, every
  admissible ideal $I \subseteq S \otimes_{\kk} K$ with Hilbert
  function $h$ has codimension $c$.
\end{lemma}

\begin{proof}
  Let $K$ be a Noetherian $\kk$-algebra and let $I \subseteq S
  \otimes_\kk K$ be an admissible ideal with Hilbert function $h
  \colon A \to \NN$.  By restricting to the torsion-free component of
  $A$ and the induced Hilbert function, it is enough to prove the
  result when $A = \ZZ^r$.  Suppose that $P \in
  \operatorname{Ass}(I)$.  We first claim that $ \mathfrak{p} := P
  \cap K$ is a minimal prime ideal in $K$.  Since $P$ is an associated
  prime of $I$, there exists $f \in R := S \otimes_\kk K$ such that $P
  = (I : f)$, so $lf \in I$ for all $l \in \mathfrak{p}$.  Since
  $(R_{\mathfrak{p}} / I_{\mathfrak{p}})_a$ is a free
  $K_{\mathfrak{p}}$-module for all $a \in \ZZ^r$, we have either $f/1
  = 0$ or $l/1 = 0$ in $R_{\mathfrak{p}} / I_{\mathfrak{p}}$.  The
  first possibility would contradict $\mathfrak{p} = (I:f) \cap K$, so
  there is $l' \in K \setminus \mathfrak{p}$ with $l'l = 0 \in K$.
  Hence all primes in $\Spec(K)$ contained in $\mathfrak{p}$ must
  contain $l$.  Because $l$ was an arbitrary element of
  $\mathfrak{p}$, we deduce that $\mathfrak{p}$ is a minimal prime.

  The codimension of $I$ in $R$ is the minimum of the codimensions of
  prime ideals in $R$ containing $I$.  If $P$ is a minimal prime ideal
  containing $I$, then $P \in \operatorname{Ass}(I)$.  Since
  $\mathfrak{p} = P \cap K$ is minimal in $\Spec(K)$, all prime ideals
  in $R$ contained in $P$ also intersect $K$ in $\mathfrak{p}$, so
  $\codim(P,R) = \codim( P_{\mathfrak{p}}, R_{\mathfrak{p}}) = \codim(
  P_{\mathfrak{p}}/\mathfrak{p} R_{\mathfrak{p}},
  k(\mathfrak{p})[\bm{x}])$ where $k(\mathfrak{p}) :=
  K_{\mathfrak{p}}/\mathfrak{p}K_{\mathfrak{p}}$ is the residue field
  at $\mathfrak{p}$.  Applying this to a prime ideal $P$ satisfying
  $\codim(I,R) = \codim(P,R)$, we see that $\codim(I,R) =
  \codim(I_{\mathfrak{p}}, R_{\mathfrak{p}}) = \codim(
  I_{\mathfrak{p}}/\mathfrak{p} R_{\mathfrak{p}},
  k(\mathfrak{p})[\mathbf{x}])$.  Since $I_{\mathfrak{p}}/\mathfrak{p}
  R_{\mathfrak{p}}$ is an admissible ideal in
  $k(\mathfrak{p})[\bm{x}]$ with Hilbert function $h \colon A \to
  \NN$, the proof reduces to the case in which $K$ is a field.

  In this case, we have $\codim(I,R) = \dim R - \dim I = N - \dim
  \initial(I)$ for any monomial initial ideal $\initial(I)$ of $I$.
  Therefore, it suffices to observe that the dimension of a monomial
  ideal $M$ is determined by its Hilbert function with respect to a
  $\ZZ^r$-grading.  For any $a \in \ZZ^r$, consider the function
  $\overline{h}_a: \NN \to \NN$ defined by $\overline{h}_a(n) :=
  h(na)$.  By combining Theorem~1 in \cite{Sturmfels} with an
  appropriate Stanley decomposition of $M$ (cf. \cite{MS}*{\S3}), we
  see that the function $\overline{h}_a$ agrees with a quasipolynomial
  of degree $d_a$ for $n \gg 0$ and the dimension of $M$ is $r +
  \max\{d_a : a \in \ZZ^r \}$.
\end{proof}

The following theorem is the key result in this section.

\begin{theorem}
  \label{t:factoring}
  Let $H$ be a connected component of $\Hilb_S^h$.  There exists a
  Hilbert function $h' \colon A \to \NN$ such that $H$ is isomorphic
  to $X \times H'$, where $X$ is either $\PP^m$ or $\AA^m$ for some $m
  \in \NN$, $H'$ is a connected component of $\Hilb_S^{h'}$, and $H'$
  parametrizes admissible ideals with codimension greater than one.
\end{theorem}

To establish this decomposition, we use the associated functors of
points; see \cite{EisenbudHarris}*{\S VI}.  Let $\mathbf{h}_Z$ be the
functor of points determined by a scheme $Z$.  For a $\kk$-algebra
$K$, we have $\mathbf{h}_Z(K) := \Hom(\Spec(K), Z)$.  From this point
of view, a morphism of schemes $Z \to Z'$ is equivalent to a natural
transformation $\mathbf{h}_Z \to \mathbf{h}_{Z'}$ of functors.  Since
the schemes in Theorem~\ref{t:factoring} are all locally Noetherian
over $\kk$, we may assume that their associated functors of points map
from the category of Noetherian $\kk$-algebras to the category of
sets.

By definition \cite{HaimanSturmfels}*{\S1}, the scheme $\Hilb_S^h$
represents the Hilbert functor $\bHilb_S^h$.  Recall that a
homogeneous ideal $I$ in $K \otimes_\kk S$ is admissible if, for all
$a \in A$, the $K$-module $(K \otimes_\kk S_a)/I_a$ is a locally free
of constant rank on $\Spec(K)$.  For a $\kk$\nobreakdash-algebra $K$,
$\bHilb_S^h(K)$ is the set of all admissible ideals $I$ in $K
\otimes_\kk S$ with Hilbert function $h \colon A \to \NN$.

\begin{proof}[Proof of Theorem~\ref{t:factoring}]
  Consider the ideal sheaf $\mathscr{I}$ on $H \times \AA^N$ which
  defines the universal admissible family over $H$ with Hilbert
  function $h \colon A \to \NN$.  If $\mathscr{I}$ is zero, then the
  theorem is trivially true, so we may assume that $\mathscr{I} \neq
  0$.  Let $\mathscr{J}$ be the intersection of the codimension-one
  primary components of $\mathscr{I}$.  Since $H$ is connected,
  Lemma~\ref{l:alg} shows that $\mathscr{J}$ and $\mathscr{Q} :=
  (\mathscr{I} : \mathscr{J})$ are admissible.  Let $h' \colon A \to
  \NN$ and $h'' \colon A \to \NN$ be the Hilbert functions associated
  to $\mathscr{Q}$ and $\mathscr{J}$ respectively.  Lemma~\ref{l:alg}
  also shows that $\mathscr{J}$ is locally principal over $H$, so
  $h''$ is the Hilbert function of some principal $S$-ideal.  The
  degree of the local generator for $\mathscr{J}$ is constant, because
  $H$ is connected.  By combining these observations with
  Lemma~\ref{l:principal}, we see that $X := \Hilb_S^{h''}$ is
  isomorphic to either $\PP^m$ or $\AA^m$ for an appropriate $m \in
  \NN$.

  We next define a natural transformation $\Phi \colon \mathbf{h}_H
  \to \mathbf{h}_X \times \bHilb_S^{h'}$.  Let $K$ be a Noetherian
  $\kk$-algebra and set $R := K \otimes_\kk S$.  Given an $R$-ideal
  $I$ corresponding to a $K$-valued point of $H$, there is a map
  $\Spec(K) \to H$ such that $I$ is the pull-back of $\mathscr{I}$.
  Using this map to pullback $\mathscr{J}$ and $\mathscr{Q}$, we
  obtain ideals $J \in \mathbf{h}_X(K)$ and $Q \in \bHilb_S^{h'}(K)$.
  Set $\Phi(I) := (J,Q)$.  Let $H'$ be the connected component of
  $\Hilb_S^{h'}$ containing the image of $\Phi$, so $\Phi \colon
  \mathbf{h}_H \to \mathbf{h}_X \times \mathbf{h}_{H'}$.

  To construct the inverse of $\Phi$, consider $K$-valued points of
  $X$ and $H'$ corresponding to $R$\nobreakdash-ideals $J'$ and $Q'$
  respectively.  Our choice of Hilbert functions $h', h'' : A \to \NN$
  together with Lemma~\ref{l:codim} show that $Q'$ has codimension
  greater than one and $J'$ has codimension at most one.  The proof of
  Lemma~\ref{l:principal} establishes that $J'$ is a locally principal
  $K$-module, so $J'_{\mathfrak{p}} = \langle f' \rangle$ where $f'$
  is a homogeneous nonzerodivisor of degree $d \in A$.  Set $I' := J'
  Q'$.  We claim that $(I':J') = Q'$.  It suffices to regard these
  ideals as $K$-modules and work locally.  Suppose that
  $Q'_{\mathfrak{p}} = \langle f_1, \dotsc, f_\ell \rangle$, so that
  $I'_{\mathfrak{p}} = \langle f' f_1, \dotsc, f' f_\ell \rangle$.  If
  $g \in (I'_{\mathfrak{p}} : f' )$ then $g f' = \sum g_i f' f_i$ for
  some $g_i \in R_{\mathfrak{p}}$, so $f' ( g - \sum g_i f_i) = 0$ and
  thus $g \in Q'_{\mathfrak{p}}$.  The other inclusion is immediate,
  so we have $(I' : J') = Q'$.  Thus multiplication by $f'$ gives the
  short exact sequence $0 \to \left(
    R_{\mathfrak{p}}/Q'_{\mathfrak{p}} \right)_{a-d} \to \left(
    R_{\mathfrak{p}}/I'_{\mathfrak{p}} \right)_a \to \left(
    R_{\mathfrak{p}}/J'_{\mathfrak{p}} \right)_a \to 0$.  It follows
  that $I'$ is admissible with Hilbert function $h \colon A \to \NN$.
  The map $(J',Q') \mapsto J'Q'$ then defines a natural transformation
  $\Psi \colon \mathbf{h}_X \times \mathbf{h}_{H'} \rightarrow
  \bHilb_S^h$.  If $I$ is a $K$-valued point of $H$, then
  Lemma~\ref{l:alg} implies that $I = JQ$, where $\Phi(I) = (J,Q)$, so
  $I$ lies in the image of $\Psi$.  Therefore, the unique connected
  component of $\Hilb^h_S$ containing the image of $\Psi$ is $H$, and
  $\Psi \colon \mathbf{h}_X \times \mathbf{h}_{H'} \rightarrow
  \mathbf{h}_H$.

  To finish the proof, we observe that $\Phi$ and $\Psi$ are mutually
  inverse.  In the last paragraph we showed that $\Psi \circ \Phi$ is
  the identity on $H$, so it suffices to check that if $J'$ and $Q'$
  correspond to $K$-valued points of $X$ and $H'$, then $J'$ is the
  codimension-one equidimensional part of $J' Q'$.  The fact that $Q'
  = (J' Q' : J')$ then follows as above.  Again it suffices to work
  locally on $\Spec(K)$.  Since $(I_{\mathfrak{p}}':f') =
  Q'_{\mathfrak{p}}$ has codimension greater than one by
  Lemma~\ref{l:codim}, $f'$ lies in the codimension-one part of
  $I'_{\mathfrak{p}}$.  If $f'$ did not generate $I'_{\mathfrak{p}}$
  there would be a nonunit common divisor of every generator of
  $Q'_{\mathfrak{p}}$, which would contradict
  $\codim(Q'_{\mathfrak{p}}, R_{\mathfrak{p}}) > 1$.  Hence, $\Phi
  \circ \Psi$ is the identity on $\mathbf{h}_X \times
  \mathbf{h}_{H'}$.
\end{proof}

\begin{example}
  Suppose that $A = \ZZ$, $S = \kk[x,y]$, $\deg(x) = 1$ and $\deg(y) =
  -1$.  Let $h \colon A \to \NN$ be the Hilbert function of the ideal
  $I = \langle x^4y^3, x^3y^4, x^2y^5 \rangle = \langle x^2y^3 \rangle
  \cdot \langle x^2, xy, y^2 \rangle$.  Since Theorem~\ref{t:main}
  establishes that $\Hilb_S^h$ is irreducible, it follows from
  Theorem~\ref{t:factoring} that $\Hilb_S^h \cong \Hilb_S^{h''} \times
  \Hilb_S^{h'} \cong \AA^2 \times \Hilb_S^{h'}$ where $h'' \colon A
  \to \NN$ is the Hilbert function of the ideal $J = \langle x^2y^3
  \rangle$ and $h' \colon A \to \NN$ is the Hilbert function of the
  ideal $Q = \langle x^2, xy, y^2 \rangle$.  Since $S_{-1}$ has
  $\kk$-basis $\{ y, xy^2, x^2y^3, \dotsc \}$, $\Hilb_S^{h''}$
  parametrizes all ideals $\langle x^2y^3 + c_1 xy^2 +c_2y \rangle$
  with $c_1,c_2 \in \kk$.
\end{example}

\section{Rationally Chain Connected}
\label{s:connected}

In this section, we prove that $\Hilb_S^h$ is rationally chain
connected when $\kk$ is a field, $S = \kk[x,y]$, and $|h| := \sum_{a
  \in A} h(a) < \infty$.  Indeed, we show that there exists a
distinguished monomial ideal in $S$, called the lex-most ideal, and a
finite chain of irreducible rational curves on $\Hilb_S^h$ connecting
any point to the point corresponding to this lex-most ideal.  The key to
exhibiting these curves is a combinatorial model for the tangent space
to $\Hilb_S^h$ at a point corresponding to a monomial ideal.

Consider a monomial ideal $M$ in $S$ with Hilbert function $h \colon A
\to \NN$ and let the monomials $x^{p_0}y^{q_0}$, $x^{p_1}y^{q_1}$,
$\dotsc$, $x^{p_n}y^{q_n}$ be the minimal generators of $M$ where $p_0
> \dotsb > p_n \geq 0$ and $0 \leq q_0 < \dotsb < q_n$.  The ideal $M$
has finite colength if and only if $p_n = 0 = q_0$.  An \define{arrow}
associated to $M$ is a triple $(i,u,v) \in \NN^3$ where $0 \leq i \leq
n$, the monomial $x^{p_i}y^{q_i}$ is a minimal generator of $M$, and
$x^uy^v$ is a standard monomial for $M$ with the same degree as
$x^{p_i}y^{q_i}$.  Because $x^uy^v \not\in M$, we must have either $u
< p_i$ or $v < q_i$.  We visualize an arrow $(i,u,v)$ as the vector
$\left[
  \begin{smallmatrix} 
    u - p_i \\
    v - q_i 
  \end{smallmatrix} \right]$ originating at the $(p_i,q_i)$-cell and
terminating at the $(u,v)$-cell; see Figure~\ref{f:arrows}.

\begin{remark}
  Despite similar nomenclature, our definition of an arrow is
  different from \cite{HaimanCatalan}*{Proposition~2.4},
  \cite{Huibregtse}*{\S2} and \cite{MillerSturmfels}*{\S18.2}.  In
  these sources, an `arrow' refers to an equivalence class of vectors;
  the equivalence relation arises from certain horizontal and vertical
  translations.  By fixing the tails of our arrows at minimal
  generators of $M$, we are choosing elements in each equivalence
  class.  The `significant arrows' defined below are in bijection with
  the nonzero equivalence classes.  This strategy follows
  \cite{Evain}*{\S2}.
\end{remark}

Arrows are classified by their direction and position of their head
relative to $M$.  To indicate the direction, we say that an arrow
$(i,u,v)$ is \define{positive} if $u > p_i$, \define{nonnegative} if
$u \geq p_i$, \define{nonpositive} if $v \geq q_i$, or \define{utterly
insignificant} if both $u < p_i$ and $v < q_i$.  The second aspect of
our classification is determined by the monomial $x^uy^v$ which we
regard as the head of the arrow $(i,u,v)$.  A nonnegative arrow
$(i,u,v)$ is \define{significant} if $i > 0$ and $x^{u+p_{i-1}-p_i}y^v
\in M$.  We denote by $T_{\geq 0}(M)$ the set of all nonnegative
significant arrows of $M$.  The subset of $T_{\geq 0}(M)$ consisting
of all positive significant arrows plays a central role and is denoted
by $T_+(M)$.  Similarly, we call a nonpositive arrow
\define{significant} if $i < n$ and $x^uy^{v-q_i+q_{i+1}} \in M$, and
denote by $T_{\leq 0}(M)$ the set of all nonpositive significant
arrows of $M$.  An arrow is simply \define{significant} if it belongs
to $T(M) := T_{\geq 0}(M) \cup T_{\leq 0}(M)$, and
\define{insignificant} otherwise.  As the notation suggests, the
significant arrows index a basis for the tangent space to $\Hilb_S^h$
at the point corresponding to $M$; see \S\ref{s:tangents}.

\begin{remark}
  \label{r:nonpositiveGrading}
  By definition, every utterly insignificant arrow is insignificant.
  If $(i,u,v)$ is an utterly insignificant arrow, then we have
  $\deg(x^{p_i-u}y^{q_i-v}) = 0 \in A$, so $\dim_\kk S_0 = \infty$.
  If $(i,u,v)$ is an arrow with either $u = p_i$ or $v = q_i$, then
  either $v < q_i$ and $\deg(y^{q_i-v}) = 0 \in A$ or $u < p_i$ and
  $\deg(x^{p_i-u}) = 0 \in A$.  In either case one variable has
  torsion degree, which also implies that $\dim_\kk S_0 = \infty$.
\end{remark}

Next, we associate an irreducible rational curve on $\Hilb_S^h$ to each
positive significant arrow $\alpha := (k, \ell+p_k, m+q_k) \in
T_{+}(M)$.  To describe this curve, we define the
$\alpha$-\define{edge ideal} to be
\begin{equation}
  \label{eq:edge}
  I_{\alpha}(t) := \langle x^{p_i}y^{q_i} : 0 \leq i < k \rangle +
  \langle x^{p_i}y^{q_i} - t x^{\ell+p_i}y^{m+q_i} : k \leq i \leq n
  \rangle
\end{equation}
where $t \in \kk$.  By construction, the $S$-ideal $I_{\alpha}(t)$ is
homogeneous with respect to the $A$\nobreakdash-grading and $M =
I_\alpha(0)$.  We occasionally regard $I_{\alpha}(t)$ as a family of
ideals over the base $\AA^1 = \Spec(\kk[t])$.

\begin{example}
  \label{e:arrows}
  If $A = 0$ and $M = \langle x^4, x^2y, y^2 \rangle$, then 
  \begin{align*}
    T_{\geq 0}(M) &= \{ (1,3,0), (1,2,0), (2,3,0), (2,2,0), (2,1,1),
    (2,0,1) \} \\
    T_{\leq 0}(M) &= \{ (0,3,0), (0,2,0), (0,1,1), (0,0,1), (1,1,1),
    (1,0,1) \} \\
    T_+(M) &= \{ (1,3,0), (2,3,0), (2,2,0), (2,1,1) \} \, .
  \end{align*}
  The insignificant arrows are $(0,0,0)$, $(0,1,0)$, $(1,0,0)$,
  $(1,1,0)$, $(2,0,0)$ and $(2,1,0)$.  If $\alpha = (1,3,0) \in
  T_+(M)$ then $k = 1$, $\ell = 1$, $m = -1$ and $I_\alpha(t) =
  \langle x^4, x^2y-tx^3, y^2-txy \rangle$.  The arrows $(1,3,0) \in
  T_+(M)$, $(0,0,1) \in T_{\leq 0}(M)$ and $(2,0,1) \in T_{\geq
    0}(M)$ are pictured in Figure~\ref{f:arrows}. \hfill $\diamond$

  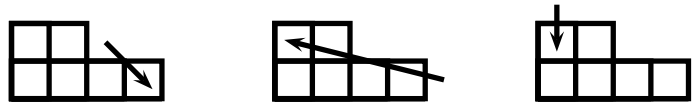
\begin{figure}[!ht]
    \psset{unit=0.5cm}
    \begin{pspicture}(0,0)(18,2.6)
      \psline[linewidth=2pt]{-}%
      (0,0)(4,0)(4,1)(0,1)(0,0)(4,0)
      \psline[linewidth=2pt]{-}%
      (0,0)(3,0)(3,1)(0,1)(0,0)(3,0)
      \psline[linewidth=2pt]{-}%
      (0,0)(2,0)(2,2)(0,2)(0,0)(2,0)
      \psline[linewidth=2pt]{-}%
      (0,0)(1,0)(1,2)(0,2)(0,0)(1,0)
      \psline[linewidth=2pt]{->}(2.5,1.5)(3.75,0.25)
      \psline[linewidth=2pt]{-}%
      (7,0)(11,0)(11,1)(7,1)(7,0)(11,0)
      \psline[linewidth=2pt]{-}%
      (7,0)(10,0)(10,1)(7,1)(7,0)(10,0)
      \psline[linewidth=2pt]{-}%
      (7,0)(9,0)(9,2)(7,2)(7,0)(9,0)
      \psline[linewidth=2pt]{-}%
      (7,0)(8,0)(8,2)(7,2)(7,0)(8,0)
      \psline[linewidth=2pt]{->}(11.5,0.5)(7.25,1.5625)
      \psline[linewidth=2pt]{-}%
      (14,0)(18,0)(18,1)(14,1)(14,0)(18,0)
      \psline[linewidth=2pt]{-}%
      (14,0)(17,0)(17,1)(14,1)(14,0)(17,0)
      \psline[linewidth=2pt]{-}%
      (14,0)(16,0)(16,2)(14,2)(14,0)(16,0)
      \psline[linewidth=2pt]{-}%
      (14,0)(15,0)(15,2)(14,2)(14,0)(15,0)
      \psline[linewidth=2pt]{->}(14.5,2.5)(14.5,1.25)
    \end{pspicture}
    \caption{     
      \label{f:arrows} 
      Three significant arrows for $\langle x^4, x^2y, y^2 \rangle$.}
  \end{figure}
\end{example}

The next result justifies our choice of generators for
$I_{\alpha}(t)$.  We write $\xel$ for the lexicographic monomial order
on $S = \kk[x,y]$ with $x < y$, and $\delta_{i,j}$ is the Kronecker
delta.

\begin{lemma}
  \label{l:gb}
  If $\alpha = (k, \ell+p_k, m+q_k) \in T_+(M)$, then the defining
  generators of $I_\alpha(t)$ form a minimal Gr\"obner basis with
  respect $\xel$ and $M = \initial_{\xel} \bigl( I_\alpha(t) \big)$.
  Moreover, there is an index $\sigma$ such that $0 \leq \sigma < k$,
  $\ell+p_{k-1} \geq p_\sigma$, $m+q_k \geq q_\sigma$, and the
  syzygies of $I_\alpha(t)$ are generated by $y^{-q_{i-1}+q_i}
  \mathbf{e}_{i-1} - x^{p_{i-1}-p_i} \mathbf{e}_i -
  \delta_{i,k}tx^{\ell+p_{k-1}-p_\sigma}y^{m+q_k-q_\sigma}
  \mathbf{e}_\sigma$ for $1 \leq i \leq n$ where $\mathbf{e}_0,
  \dotsc, \mathbf{e}_n$ is the standard basis for the $A$-graded free
  $S$-module $\bigoplus_{i=0}^{n} S\big( - \deg(x^{p_i}y^{q_i})
  \bigr)$.
\end{lemma}

\begin{proof}
  Since the minimal generators of $M$ are the initial terms with
  respect to $\xel$ of the defining generators of $I_\alpha(t)$, it
  suffices to show that these generators form a Gr\"{o}bner basis.  By
  Buchberger's criterion \cite{Eisenbud}*{Exercise~15.19}, we need
  only prove that certain S-polynomials reduce to zero modulo the
  generators of $I_\alpha(t)$, namely those for pairs of generators
  corresponding to the minimal syzygies of $M$.  For any monomial
  ideal in the ring $S = \kk[x,y]$, Proposition~3.1 of
  \cite{MillerSturmfels} shows that the minimal syzygies correspond to
  adjacent pairs of minimal generators.  The S-polynomial between any
  pair of monomials is always zero.  For any pair of adjacent binomial
  generators in $I_\alpha(t)$, the S-polynomial is
  \[
  y^{-q_{i-1}+q_i}(x^{p_{i-1}}y^{q_{i-1}} -
  tx^{\ell+p_{i-1}}y^{m+q_{i-1}}) - x^{p_{i-1}-p_i}(x^{p_i}y^{q_i} -
  tx^{\ell+p_i}y^{m+q_i}) = 0 \, ,
  \]
  where $k < i \leq n$.  Hence, the final S-polynomial to examine is
  \[
  y^{-q_{k-1}+q_k}(x^{p_{k-1}}y^{q_{k-1}}) -
  x^{p_{k-1}-p_k}(x^{p_k}y^{q_k} - tx^{\ell+p_k} y^{m+q_k}) =
  tx^{\ell+p_{k-1}}y^{m+q_k} \, .
  \]
  Since $\alpha \in T_+(M)$, we have $\ell > 0$ and $m < 0$, so the
  monomials $x^{p_i}y^{q_i}$ for $i \geq k$ cannot divide
  $x^{\ell+p_{k-1}}y^{m+q_k}$.  However, $\alpha \in T_+(M)$ implies
  that $x^{\ell+p_{k-1}}y^{m+q_k} \in M$, so
  $x^{\ell+p_{k-1}}y^{m+q_k}$ is divisible by at least one of the
  monomials $x^{p_{\sigma}}y^{q_{\sigma}}$ for $\sigma < k$.
  Therefore, the final S-polynomial reduces to zero modulo the
  generators of $I_\alpha(t)$.  The assertion about the syzygies of
  $I_\alpha(t)$ then follows from \cite{Eisenbud}*{Theorem~15.10}.
\end{proof}

\begin{example}
  \label{e:arrowsSyz}
  If $A = 0$, $M = \langle x^4, x^2y, y^2 \rangle$ and $\alpha =
  (1,3,0) \in T_+(M)$ as in Example~\ref{e:arrows}, then the syzygies
  of the $\alpha$-edge ideal $I_{\alpha}(t)$ are generated by
  $y\mathbf{e}_0 - x^2 \mathbf{e}_1 - xt \mathbf{e}_0$ and $y
  \mathbf{e}_1 - x^2 \mathbf{e}_2$; here $\sigma = 0$. \hfill
  $\diamond$
\end{example}

\begin{example}
  \label{e:syz}
  If $A = 0$, $M = \langle x^7, x^6y, x^5y^2, x^4y^3, x^2y^4, y^6
  \rangle$, and $\alpha = (4,3,2) \in T_+(M)$, then the syzygies of
  the $\alpha$-edge ideal $I_{\alpha}(t) := \langle x^7, x^6y, x^5y^2,
  x^4y^3, x^2y^4-tx^3y^2, y^6-txy^4 \rangle$ are generated by
  $y\mathbf{e}_0-x\mathbf{e}_1$, $y\mathbf{e}_1-x\mathbf{e}_2$,
  $y\mathbf{e}_2-x\mathbf{e}_3$, $y\mathbf{e}_3-x^2\mathbf{e}_4
  -t\mathbf{e}_2$ and $y^2\mathbf{e}_4-x^2\mathbf{e}_5$; the index
  $\sigma$ is $2$. \hfill $\diamond$
\end{example}

Following \cite{Yameogo}*{Th\'eor\`eme~3.2} (also see
\cite{Evain}*{Definition~17}), we introduce a partial order on the set
of all monomial ideals with a given Hilbert function.  Given two
monomial ideals $M$ and $M'$ with the same Hilbert function, we say
$M' \succcurlyeq M$ if, for all monomials $x^ry^s \in S$, the number
of standard monomials for $M'$ with degree equal to $\deg(x^ry^s)$
lexicographically less than or equal to $x^ry^s$ is at least the
number of standard monomials for $M$ with degree equal to
$\deg(x^ry^s)$ lexicographically less than or equal to $x^ry^s$.  The
reflexivity, antisymmetry and transitivity of $\succcurlyeq$ follow
from the properties of the canonical order on $\NN$.  Given a Hilbert
function $h \colon A \to \NN$, let $\mathcal{P}_h$ denote the poset of
all monomial ideals with Hilbert function $h$.  If $M' \neq M$ and $M'
\succcurlyeq M$, then we write simply $M' \succ M$.

\begin{remark}
  Following \cite{MillerSturmfels}*{\S3.1}, we identify a monomial
  ideal $M$ in $S = \kk[x,y]$ with its staircase diagram.  When the
  Hilbert function $h \colon A \to \NN$ of $M$ satisfies $|h| <
  \infty$, the staircase diagram of $M$ is a Young diagram (in the
  French tradition).  Hence, the rows of the diagram correspond to the
  parts of a partition of $|h|$.  When $A = 0$, and $|h| < \infty$,
  the partial order $\succ$ is the dominance order applied to 
  the conjugate partitions.
\end{remark}

\begin{example}
  \label{e:poset}
  Suppose that $A = \ZZ$ and $\deg(x) = 1 = \deg(y)$.  Among the
  eleven monomial ideals of colength six in $S$, there are exactly six
  monomial ideals with Hilbert function given by $h(0) = 1$, $h(1) =
  2$, $h(2) = 2$, $h(3) = 1$ and $h(a) = 0$ for all $a \geq 3$.
  Figure~\ref{f:poset} illustrates the Hasse diagram for the poset
  $\mathcal{P}_h$. \hfill $\diamond$

  \begin{figure}[!ht]
    \psset{unit=0.25cm}
    \begin{pspicture}(0,0)(15,21)
      \psline[linewidth=2pt]{-}%
      (5,0)(9,0)(9,1)(7,1)(7,2)(5,2)(5,0)(8,0)(8,1)(5,1)(5,2)%
      (6,2)(6,0)(7,0)(7,1)
      \psline[linewidth=2pt]{-}%
      (0,5.5)(3,5.5)(3,7.5)(0,7.5)(0,5.5)(1,5.5)(1,7.5)(2,7.5)%
      (2,5.5)(3,5.5)(3,6.5)(0,6.5)
      \psline[linewidth=2pt]{-}%
      (11,5)(15,5)(15,6)(12,6)(12,8)(11,8)(11,5)(14,5)(14,6)(11,6)%
      (11,7)(12,7)(12,5)(13,5)(13,6)
      \psline[linewidth=2pt]{-}%
      (0,11)(3,11)(3,12)(1,12)(1,15)(0,15)(0,11)(2,11)(2,12)(0,12)%
      (0,13)(1,13)(1,11)(1,14)(0,14)
      \psline[linewidth=2pt]{-}%
      (11.5,11.5)(13.5,11.5)(13.5,14.5)(11.5,14.5)(11.5,11.5)%
      (12.5,11.5)(12.5,14.5)(11.5,14.5)(11.5,13.5)(13.5,13.5)%
      (13.5,12.5)(11.5,12.5)
      \psline[linewidth=2pt]{-}%
      (6,17)(8,17)(8,19)(7,19)(7,21)(6,21)(6,17)(7,17)%
      (7,19)(6,19)(6,18)(8,18)(6,18)(6,20)(7,20)
      \psline[linewidth=2pt]{-}(7,16.5)(2,15)
      \psline[linewidth=2pt]{-}(7,16.5)(12.5,15)
      \psline[linewidth=2pt]{-}(2,10.5)(2,8)
      \psline[linewidth=2pt]{-}(2,10.5)(12.5,8.5)
      \psline[linewidth=2pt]{-}(12.5,11)(12.5,8.5)
      \psline[linewidth=2pt]{-}(12.5,11)(2,8)
      \psline[linewidth=2pt]{-}(2,5)(7,2.5)
      \psline[linewidth=2pt]{-}(12.5,4.5)(7,2.5)
    \end{pspicture}
    \caption{
      \label{f:poset}
      Hasse diagram for the poset in Example~\ref{e:poset}.
    }
  \end{figure}
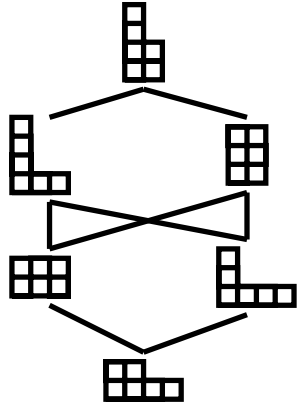
\end{example}

The next lemma records a well-known geometric interpretation for
Gr\"obner bases.  We write $\lex$ for the lexicographic monomial order
on $S = \kk[x,y]$ with $x > y$.

\begin{lemma}
  \label{l:curve}
  Given an $S$-ideal $I$ corresponding to a point on $\Hilb_S^h$, the
  Gr\"obner degenerations of $I$ with respect to $\lex$ and $\xel$
  describe an irreducible rational curve on $\Hilb_S^h$ containing the
  points corresponding to $I$, $\initial_{\xel}(I)$ and
  $\initial_{\lex}(I)$.
\end{lemma}

\begin{proof}
  Proposition~15.16 in \cite{Eisenbud} gives a weight vector $w \in
  \ZZ^2$ such that $\initial_{w}(I) = \initial_{\xel}(I)$ and
  $\initial_{-w}(I) = \initial_{\lex}(I)$.  Applying Theorem~15.17 in
  \cite{Eisenbud}, we obtain a flat family of admissible ideals over
  $\PP^1$ in which the fibres over $0$, $1$ and $\infty$ are
  $\initial_{\xel}(I)$, $I$ and $\initial_{\lex}(I)$ respectively.
  Since $\Hilb_S^h$ is a fine moduli space, this family gives a map
  from $\PP^1$ to $\Hilb_S^h$ whose image contains the points
  corresponding to $I$, $\initial_{\xel}(I)$ and $\initial_{\lex}(I)$.
\end{proof}

We now apply Lemma~\ref{l:curve} to describe the irreducible rational
curve on $\Hilb_S^h$ associated to the positive significant arrow
$\alpha \in T_{+}(M)$.

\begin{proposition} 
  \label{p:edgeideal}
  Let $M$ be a monomial ideal in $S$.  If $\alpha \in T_+(M)$ and $t
  \neq 0$, then $I_\alpha(t)$ has exactly two initial ideals, namely
  $M = \initial_{\xel}\bigl( I_\alpha(t) \bigr)$ and $M' :=
  \initial_{\lex}\bigl( I_\alpha(t) \bigr)$.  Moreover, we have $M'
  \succ M$ and, on $\Hilb_S^h$, the points corresponding to $M$ and
  $M'$ lie on an irreducible rational curve.
\end{proposition}

\begin{proof}
  Let $\alpha = (k,\ell+p_k, m+q_k)$ and consider the vector $[
  \begin{smallmatrix} 
    \ell \\ m 
  \end{smallmatrix} 
  ] \in \ZZ^2$.  By construction, the ideal $I_\alpha(t)$ is
  homogeneous with respect to the induced $(\ZZ^2/\ZZ [
  \begin{smallmatrix} 
    \ell \\ m 
  \end{smallmatrix} 
  ] )$-grading of $S$.  A polynomial in $S$ that is homogeneous with
  respect to this grading has only two possible initial terms.
  Moreover, these two initial terms are given by $\lex$ and $\xel$.
  Hence, there are only two equivalence classes of monomial orders
  with respect to $I_\alpha(t)$.  It follows that $I_\alpha(t)$ has at
  most two distinct initial ideals.  Lemma~\ref{l:gb} establishes that
  $M = \initial_{\xel}\bigl( I_\alpha(t) \bigr)$.  Lemma~\ref{l:curve}
  shows that the Gr\"obner degenerations of $I_{\alpha}(t)$ give with
  an irreducible rational curve on $\Hilb_S^h$ containing the points
  corresponding to $M$, $I_{\alpha}(t)$ and $M'$, so it remains to
  show that $M' \succ M$.

  Since $t \neq 0$, we know $I_\alpha(t) \neq M$, so $M =
  \initial_{\xel}\bigl( I_\alpha(t) \bigr)$ implies $M \neq M'$.
  Suppose that $M' \nsucc M$; this means there exists a monomial
  $x^ry^s \in S$ such that the number of standard monomials for $M'$
  with degree equal to $\deg(x^ry^s)$ lexicographically less than or
  equal to $x^ry^s$ is strictly less than the number of such standard
  monomials for $M$.  Choosing $x^ry^s$ to be the lexicographically
  smallest monomial with this property guarantees that $x^ry^s \in
  M'$, $x^ry^s \not\in M$ and each monomial lexicographically less
  than or equal to $x^ry^s$ with degree equal to $\deg(x^ry^s)$ is
  either in both of $M'$ and $M$ or in neither monomial ideal.
  Because $I_\alpha(t)$ is a binomial ideal, the remainder of $x^ry^s$
  on division by the Gr\"{o}bner basis for $I_\alpha(t)$ with respect
  to $\lex$ is a monomial, say $x^uy^v$.  Since $M' =
  \initial_{\lex}\bigl( I_\alpha(t) \bigr)$, we have $x^uy^v \not\in
  M'$, so $x^ry^s \neq x^uy^v$.  Hence, $x^ry^s -t^{\lambda} x^uy^v
  \in I_\alpha(t)$ for some $\lambda >0$ and $x^ry^s \lex x^uy^v$
  which implies that $x^uy^v \not\in M$.  But this means
  $\initial_{\xel}(x^ry^s - t^{\lambda} x^uy^v) \not\in M =
  \initial_{\xel}\bigl( I_\alpha(t) \bigr)$ which is a contraction.
\end{proof}

\begin{example}
  If $A = 0$, $M = \langle x^4, x^2y, y^2 \rangle$ and $\alpha =
  (1,3,0) \in T_+(M)$ as in Example~\ref{e:arrows}, then we have
  $I_\alpha(t) = \langle x^4,x^2y-tx^3, y^2-txy \rangle$ and its
  initial ideals are $M = \initial_{\xel}\bigl( I_\alpha(t) \bigr)$
  and $M' := \langle x^3, xy, y^4 \rangle = \initial_{\lex}\bigl(
  I_\alpha(t) \bigr)$.  The map $[ z_0 \colon z_1 ] \mapsto \langle
  x^4, z_0x^2y- z_1 x^3, z_0y^2-z_1xy, y^4 \rangle$ induces a morphism
  from $\PP^1$ to the appropriate multigraded Hilbert scheme.  In
  particular, we have $[1 \colon 0] \mapsto M$, $[0 \colon 1] \mapsto
  M'$, and $[1 \colon t] \mapsto I_\alpha(t)$. \hfill $\diamond$
\end{example}

For a Hilbert function $h \colon A \to \NN$ satisfying $|h| := \sum_{a
  \in A} h(a) < \infty$, $|h|$ equals the colength of the ideals
parametrized by $\Hilb_S^h$.

\begin{proposition} 
  \label{p:lexmost}
  For a Hilbert function $h \colon A \to \NN$ with $|h| < \infty$,
  there exists a unique monomial ideal $L_h \in \mathcal{P}_h$ such
  that $T_+(L_h) = \varnothing$.  Thus, the poset $\mathcal{P}_h$ has
  a unique maximal element.
\end{proposition}

\noindent
We call the monomial ideal $L_h$ of Proposition~\ref{p:lexmost} the
\define{lex-most} ideal with Hilbert function $h$.

\begin{proof}[Proof of Existence]
  Asserting $h \colon A \to \NN$ is a Hilbert function means that
  there exists an ideal $I$ with Hilbert function equal to $h$.
  Hence, $M = \initial_{\xel}(I)$ is a monomial ideal with Hilbert
  function $h$.  There are only finitely many monomial ideals with
  Hilbert function $h$, so the poset $\mathcal{P}_h$ has at least one
  maximal element.  Proposition~\ref{p:edgeideal} shows that $M \in
  \mathcal{P}_h$ is not maximal when $T_+(M) \neq \varnothing$.
  Therefore, there is at least one monomial ideal $L_h \in
  \mathcal{P}_h$ with $T_+(L_h) = \varnothing$.
\end{proof}

\begin{proof}[Proof of Uniqueness]
  We induct on $|h|$.  Proposition~\ref{p:edgeideal} shows that only
  monomial ideals with no positive significant arrows can be maximal
  elements of $\mathcal{P}_h$.  Suppose that the monomial ideal $M =
  \langle x^{p_0}y^{q_0}, \dotsc, x^{p_n}y^{q_n} \rangle$ is a maximal
  element of $\mathcal{P}_h$, so $T_+(M) = \varnothing$.  Since $|h| <
  \infty$, $M$ has finite colength and $p_n = 0 = q_0$.  If $|h| = 0$
  or $1$, then $\langle 1 \rangle$ or $\langle x,y \rangle$
  respectively is the unique monomial ideal in $\mathcal{P}_h$, so the
  base case of the induction holds.

  For the induction step, we examine the ideal $(M \colon y)$.  The
  minimal generators of $(M \colon y)$ are either $\langle x^{p_0},
  x^{p_1}y^{q_1-1}, \dotsc, y^{q_n-1} \rangle$ when $q_1 > 1$ or
  $\langle x^{p_1}, x^{p_2}y^{q_2-1}, \dotsc, y^{q_n-1} \rangle$ when
  $q_1 = 1$.  As a preamble, we prove that $T_+(M \colon y) =
  \varnothing$.  If there exists a pair $(x^{p_i}y^{q_i -1},
  x^uy^{v-1})$ corresponding to a positive significant arrow of $(M
  \colon y)$, then $i > 0$, $u-p_i > 0$, and $x^{u+p_{i-1}-p_i}y^{v-1}
  \in (M \colon y)$.  The definition of the ideal quotient implies
  that $x^uy^v \not \in M$, and $x^{u+p_{i-1}-p_i}y^v \in M$.  Hence,
  $(i,u,v) \in T_+(M) = \varnothing$ which is a contradiction.
  Additionally, the short exact sequence
  \[
  0 \to \frac{S}{(M:y)}\bigl( -\deg(y) \bigr) \xrightarrow{\; y \;}
  \frac{S}{M} \to \frac{S}{(x^{p_0},y)} \to 0
  \]
  implies that $|h'| = \sum_{a \in A} h'(a) = |h| - p_0 < |h|$ where
  $h' \colon A \to \NN$ is the Hilbert function of $(M \colon y)$, so
  the induction hypothesis ensures that $(M \colon y)$ is unique.
  Therefore it suffices to show that all the maximal elements of
  $\mathcal{P}_h$ contain the same power of $x$ as a minimal
  generator.

  To complete the proof, we assume that $M$ is chosen so that the
  power of the minimal generator $x^{p_0}$ is maximal among all the
  maximal elements of $\mathcal{P}_h$.  We break our analysis into two
  cases.  First, suppose that there is no standard monomial $x^uy^v$
  of $M$ with degree equal to $\deg(x^{p_0 -1}) - \deg(y)$ such that
  $x^uy^{v+1} \in M$.  It follows that 
  \[
  \hslash:= h\bigl( \deg(x^{p_0 -1}) \bigr) - h \bigl( \deg(x^{p_0
    -1}) - \deg(y) \bigr)
  \] 
  is the number of standard monomials for $M$ of degree $\deg(x^{p_0
    -1})$ that are pure powers of $x$.  Moreover, $x^{p_0 -1}$ must be
  the $\hslash^{\text{th}}$ such monomial.  For any $M' \in
  \mathcal{P}_h$, there must be at least $\hslash$ standard monomials
  of degree $\deg(x^{p_0 -1})$ that are pure powers of $x$.  As a
  result, $x^{p_0 -1}$ is standard for all $M'$.  From our choice of
  $M$, we conclude that all the maximal elements of $\mathcal{P}_h$
  contain $x^{p_0}$ as a minimal generator in this case.

  For the second case, suppose that there is a standard monomial
  $x^uy^v$ of $M$ with degree equal to $\deg(x^{p_0 -1}) - \deg(y)$
  such that $x^uy^{v+1} \in M$.  Since there exists a minimal
  generator of $M$ dividing $x^uy^{v+1}$, there is an index $i > 0$
  such that $p_i \leq u < p_{i-1}$ and $q_i = v+1$.  If $u < p_0 -1$
  then we have $(i,p_0-1+p_i-u,0) \in T_{+}(M) = \varnothing$ which is
  a contradiction.  Hence, we may assume that $u = p_0 -1$ which
  implies that $(v+1)\deg(y) = \deg(y^{v+1}) = 0$.  Now, consider a
  hypothetical monomial $x^ry^s \in S$ satisfying $\deg(x^ry^s) =
  \deg(x^{p_0 -1})$ and $r < p_0 -1$.  Since the ideal $M$ has finite
  colength, there is a $\zeta \geq v+1$ such that $x^ry^{\zeta} \in M$
  and $x^r y^{\zeta-1} \not \in M$.  Thus, there is $0 \leq \xi \leq
  v$ with $\deg(x^ry^s) = \deg(x^ry^{\zeta-\xi})$, because
  $\deg(y^{v+1}) = 0$.  If $1 \leq j \leq n$ is the index such that
  $p_j \leq r < p_{j-1}$ and $\zeta = q_j$, then we have
  $\deg(x^{p_j}y^{q_j}) = \deg(x^{p_0-1-r+p_j}y^{\xi})$, so
  $(j,p_0-1-r+p_j,\xi) \in T_+(M) = \varnothing$ which is a
  contradiction.  In other words, the hypothetical monomial $x^ry^s$
  cannot exist.  Since $h\bigl( \deg(x^{p_0-1}) \bigr) > 0$, we deduce
  that $x^{p_0-1}$ must be a standard monomial for all $M' \in
  \mathcal{P}_h$.  From our choice of $M$, we again conclude that all
  the maximal elements of $\mathcal{P}_h$ contain $x^{p_0}$ as a
  minimal generator in this case.
\end{proof}

\begin{example}
  \label{e:nonlex}
  Suppose that $A = \ZZ/3\ZZ$, $\deg(x) = 1$ and $\deg(y) = 1$.  The
  monomial ideals in $S$ with Hilbert function $h(0) = 2$, $h(1) = 3$
  and $h(2) = 1$ are $M := \langle x^5,xy,y^2 \rangle$ and $M' :=
  \langle x^2, xy, y^5 \rangle$.  The poset $\mathcal{P}_h$ is the
  chain $M' \succ M$.  Since we have $x \lex y^7$, $\deg(x) =
  \deg(y^7)$, $y^7 \in M'$ and $x \not\in M'$, it follows that the
  lex-most ideal $M'$ is not a lex-segment ideal.  See
  \cite{MillerSturmfels}*{\S2.4} for more information on lex-segment
  ideals. \hfill $\diamond$
\end{example}

\begin{example}
  \label{e:nonExtremalBetti}
  Suppose that $A = 0$. The monomial ideals in $S$ with Hilbert
  function $h(0) = 3$ are $M := \langle x^3, y \rangle$, $M' :=
  \langle x^2, xy, y^2 \rangle$ and $M'' := \langle x, y^3 \rangle$.
  The poset $\mathcal{P}_h$ is the chain $M'' \succ M' \succ M$.  The
  monomial ideal $M'$ has the largest Betti numbers rather than the
  maximal element $M''$ of $\mathcal{P}_h$.  Thus the
  analogue of the Bigatti-Hulett Theorem
  \cite{MillerSturmfels}*{Theorem~2.24} is false. \hfill $\diamond$
\end{example}

We end this section with its central result.  A scheme is rationally
chain connected if two general points can be joined by a chain of
irreducible rational curves; see \cite{Kollar}*{\S IV.3}.

\begin{theorem}
  \label{t:connected}
  If $S = \kk[x,y]$ and the Hilbert function $h \colon A \to \NN$
  satisfies $|h| < \infty$, then the points on $\Hilb_S^h$
  corresponding to monomial ideals are connected by irreducible
  rational curves associated to positive significant arrows.
  Consequently, $\Hilb_S^h$ is rationally chain connected.
\end{theorem}

\begin{proof}
  Consider a point on $\Hilb_S^h$ corresponding to a monomial ideal
  $M$.  We first exhibit a finite chain of curves associated to
  positive significant arrows connecting the points on $\Hilb_S^h$
  corresponding to $M$ and $L_h$.  If $M \neq L_h$,
  Proposition~\ref{p:lexmost} implies that there exists $\alpha \in
  T_+(M)$.  If $M' := \initial_{\xel}\bigl( I_{\alpha}(t) \bigr)$,
  then Proposition~\ref{p:edgeideal} produces an irreducible rational
  curve associated to $\alpha$ which contains the points corresponding
  $M$ and $M'$.  Proposition~\ref{p:edgeideal} also shows that $M'
  \succ M$ in $\mathcal{P}_h$.  If $M' \neq L$, then we may repeat
  these steps.  Since Proposition~\ref{p:lexmost} shows that the
  lex-most ideal $L_h$ is the unique maximal element in
  $\mathcal{P}_h$, this process terminates with a curve that contains
  the point corresponding to $L_h$.  Thus, for every pair of points on
  $\Hilb_S^h$ corresponding to monomial ideals, there is a connected
  curve containing both points for which every irreducible
  component is a rational curve associated to a positive significant
  arrow.

  For each closed point on $\Hilb_S^h$, we produce an irreducible
  rational curve containing this point and a point corresponding to a
  monomial ideal.  If the ideal $I'$ corresponds to a point on
  $\Hilb_S^h$, then Lemma~\ref{l:curve} shows that the Gr\"{o}bner
  degenerations of $I'$ give an irreducible rational curve on
  $\Hilb_S^h$ which contains the points corresponding to $I'$ and
  $\initial_{\xel}(I')$.  Therefore, for every pair of closed points
  on $\Hilb_S^h$ there is a connected curve, in which every
  irreducible component is rational, that contains both points.
\end{proof}

\section{Tangent Spaces}
\label{s:tangents}

This section relates the combinatorics of the significant arrows to
the geometry of the multigraded Hilbert scheme.  Given a monomial
ideal $M$ in $S = \kk[x,y]$ with Hilbert function $h \colon A \to \NN$
satisfying $|h| < \infty$, fix $\alpha \in T_+(M)$ and let
$I_{\alpha}(t)$ be the $\alpha$-edge ideal defined in \eqref{eq:edge}.
We prove that for all $t \in \kk$ the significant arrows of $M$ index
a basis for the tangent space to $\Hilb_S^h$ at the point
corresponding $I_{\alpha}(t)$.  To accomplish this, we first identify
the tangent space to $\Hilb_S^h$ at the point corresponding to
$I_\alpha(t)$ with an explicit linear subspace.

\begin{proposition}
  \label{p:linear}
  Let $M = \langle x^{p_0}y^{q_0}, \dotsc, x^{q_n}y^{q_n} \rangle$ be
  a monomial ideal in $S$ with Hilbert function $h \colon A \to \NN$
  and let $\alpha = (k, \ell+p_k,m+q_k)$ be a positive significant
  arrow for $M$.  The tangent space to $\Hilb_S^h$ at the point
  corresponding to $I_{\alpha}(t)$ is isomorphic to the linear
  subspace of $\AA^r := \Spec \bigl( \kk[c_{u,v}^i : \text{$(i,u,v)$
    is an arrow of $M$}] \bigr)$ cut out by the homogeneous linear
  equations
  \begin{multline}
    \label{eq:linear}
    F(i,u,v) := \textstyle\sum\limits_{\mu = 0}^{b_{u,v}}
    t^\mu \bigl( c_{u-\mu \ell, v+q_{i-1}-q_i-\mu m}^{i-1}
    - c_{u-p_{i-1}+p_i-\mu \ell, v-\mu m}^i \\ - \delta_{i,k}
    tc_{u- p_{k-1}+p_\sigma-(\mu+1)\ell,
      v-q_k+q_\sigma-(\mu+1)m}^\sigma \bigr) \, ,
  \end{multline}
  where $1 \leq i \leq n$, $x^uy^v \not\in M$, $\sigma$ is the largest
  index satisfying $0 \leq \sigma < k$, $\ell + p_{k-1} \geq p_\sigma$
  and $m + q_k \geq q_\sigma$, and $b_{u,v}$ is the largest
  nonnegative integer satisfying $x^{u-\kappa\ell}y^{v-\kappa m} \in
  M$ for all $0 < \kappa < b_{u,v}$.
\end{proposition}

\begin{remark}
  The tangent space to $\Hilb_S^h$ at the point corresponding to $M =
  I_\alpha(0)$ is cut out by $F(i,u,v) = c_{u, v+q_{i-1}-q_i}^{i-1} -
  c_{u-p_{i-1}+p_i, v}^i$ for $1 \leq i \leq n$ and $x^uy^v \not\in
  M$.  Lemma~\ref{l:gb} explains the importance of the index $\sigma$.
\end{remark}

\begin{proof}[Proof of Proposition~\ref{p:linear}]
  For simplicity, set $I := I_\alpha(t)$.  Lemma~\ref{l:gb} shows that
  the given generators of $I$ form a minimal Gr\"obner basis with
  respect to $\xel$, and Proposition~\ref{p:edgeideal} together with
  \cite{Eisenbud}*{Theorem~15.3} shows that the standard monomials of
  $M$ form a $\kk$-basis for $S/I$.  By
  \cite{HaimanSturmfels}*{Proposition~1.6}, the tangent space is
  isomorphic to $\bigl( \Hom_S(I,S/I) \bigr)_0$ where $0 \in A$.
  Given $\psi \in \bigl( \Hom_S(I,S/I) \bigr)_0$, the $i^{\text{th}}$
  generator of $I$ maps to $\sum_{u,v} c_{u,v}^{i} x^u y^v$ where
  $c_{u,v}^{i} \in \kk$ and the sum runs over all arrows for $M$ of
  the form $(i,u,v)$.  If $r_i$ is the number of such arrows and $r :=
  \sum_{i=0}^n r_i$, then each $\psi \in \bigl( \Hom_S(M,S/M)
  \bigr)_0$ produces a point $(c_{u,v}^i) \in \AA^r$.  Conversely, a
  point $(c_{u,v}^i) \in \AA^r$ defines $\varphi \in \bigl( \Hom_S(
  \bigoplus_{i=0}^n S(- \deg(x^{p_i}y^{q_i})),S/I) \bigr)_0$ by
  sending the $i^{\text{th}}$ standard basis element $\mathbf{e}_i$ of
  $\bigoplus_{i=0}^n S\bigl(- \deg(x^{p_i}y^{q_i}) \bigr)$ to
  $\sum_{u,v} c_{u,v}^i x^u y^v$; again the sum runs over all arrows
  $(i,u,v)$ of $M$.  The syzygies of $I$ determine whether $\varphi$
  restricts to $\bigl( \Hom_S(I,S/I) \bigr)_0$.  More precisely,
  Lemma~\ref{l:gb} provides a free presentation of $I$ having the form
  $S^{n} \xrightarrow{\partial} S^{n+1} \to I \to 0$.  From this, we
  obtain the exact sequence
  \[
  0 \to \Hom_S(I,S/I) \to \Hom_S(S^{n+1},S/I)
  \xrightarrow{\overline{\partial}} \Hom_S(S^{n},S/I) \, ,
  \] 
  so a point in $\AA^r$ defines an element $\bigl( \Hom_S(I,S/I)
  \bigr)_0$ if and only if $\overline{\partial}(\varphi) = \varphi
  \circ \partial = 0$.  This condition is equivalent to a system of
  homogeneous linear equations in the $c_{u,v}^i$. 
    
  We next describe this system of equations.  Since Lemma~\ref{l:gb}
  provides a generating set of syzygies for $I$, we see that $\varphi
  \in \bigl( \Hom_S(I,S/I) \bigr)_0$ if and only if we have
  \begin{multline*}
    0 = y^{-q_{i-1}+q_i} \left( \textstyle\sum_{u,v} c_{u,v}^{i-1} x^u
      y^v \right) - x^{p_{i-1} - p_i} \left( \textstyle\sum_{u,v}
      c_{u,v}^{i} x^u y^v \right) \\ - \delta_{i,k} t x^{\ell +
      p_{k-1} - p_\sigma}y^{m+q_k-q_\sigma} \left(
      \textstyle\sum_{u,v} c_{u,v}^{\sigma} x^u y^v \right) \in S/I\,
      ,
  \end{multline*}
  where the sums run over all $x^u y^v \not\in M$ and $c_{u,v}^{i} =
  0$ when the triple $(i,u,v)$ fails be an arrow.  Lemma~\ref{l:gb}
  also shows that the defining generators for $I$ form a Gr\"obner
  basis such that $\initial_{\xel}(I) = M$.  By taking the normal form
  with respect to the generators of $I$, these equations produce an
  equation for each triple $(i,u,v)$ such that $1 \leq i \leq n$ and
  $x^uy^v \not\in M$.  To be more explicit, observe that the coset in
  $S/I$ containing $x^uy^v \not\in M$ can have more than one element
  only when $x^uy^v \in \langle x^{\ell+p_j}y^{m+q_j} : j \geq k
  \rangle$.  Let $b_{u,v}$ be the largest nonnegative integer such
  that $x^{u-\kappa \ell}y^{v-\kappa m} \in M$ for $0 < \kappa \leq
  b_{u,v}$.  Since $\alpha \in T_+(M)$, we
  have $\ell > 0$, so $u-j\ell<0$ for $j \gg 0$, and thus $b_{u,v} <
  \infty$.  With this notation, the set $\{ x^{u-\mu \ell}y^{v- \mu m}
  : 0 \leq \mu \leq b_{u,v} \}$ consists of all the monomials in $S$
  which reduce to $x^uy^v$ modulo the generators of $I$.  Hence, the
  equation labelled by $(i,u,v)$ is
  \begin{multline*}
    F(i,u,v) := \textstyle\sum\limits_{\mu =0}^{b_{u,v}} t^\mu
    \bigl( c_{u- \mu \ell,v+q_{i-1}-q_i- \mu m}^{i-1} -
    c_{u-p_{i-1}+p_i- \mu \ell, v- \mu m}^{i}  \\
    - \delta_{i,k} tc_{u- p_{k-1} + p_\sigma - (\mu +1)\ell,
      v-q_k+q_\sigma-(\mu +1)m}^{\sigma} \bigr)\, .
  \end{multline*}
  It follows that the tangent space is isomorphic to the linear
  subvariety of $\AA^r$ cut out by these homogeneous linear equations.
\end{proof}

\begin{example}
  \label{e:arrowsL}
  If $A = 0$, $M = \langle x^4, x^2y, y^2 \rangle$, and $\alpha =
  (1,3,0) \in T_+(M)$ as in Example~\ref{e:arrowsSyz}, then the
  subspace in Proposition~\ref{p:linear} is cut out by:
  \begin{xalignat*}{4}
    F(1,0,0) &= 0 & F(1,1,0) &= -tc_{0,0}^{0} &
    F(1,2,0) &= -c_{0,0}^{1}-tc_{1,0}^{0} && \\
    F(1,3,0) &= -c_{1,0}^{1}-tc_{0,1}^{1} &
    F(1,0,1) &= c_{0,0}^{0} & F(1,1,1) &= c_{1,0}^{0} && \\
    F(2,0,0) &= 0  & F(2,1,0) &= 0 &
    F(2,2,0) &= - c_{0,0}^{2} && \\ F(2,3,0) &=
    tc_{2,0}^{1}+t^2c_{1,1}^{1}-c_{1,0}^{2} - tc_{0,1}^{2} &
    F(2,0,1) &= c_{0,0}^{1} & F(2,1,1) &=
    c_{1,0}^{1}+tc_{0,1}^{1} \, . && \quad\diamond
  \end{xalignat*}
\end{example}

\begin{example}
  \label{e:L}
  As in Example~\ref{e:syz}, suppose that $A = 0$, $M = \langle x^7,
  x^6y, x^5y^2, x^4y^3, x^2y^4, y^6 \rangle$, and $\alpha = (4,3,2)
  \in T_+(M)$.  Since $M$ has six generators and colength $26$, the
  linear subvariety in Proposition~\ref{p:linear} is defined by
  $(6-1)(26) = 130$ equations.  The following six equations illustrate
  some of the possibilities.
  \begin{xalignat*}{4}
    F(1,5,0) &= - c_{4,0}^{1} & 
    F(3,3,2) &= c_{3,1}^{2}+tc_{2,3}^{2}+t^2c_{1,5}^{2} -
    c_{2,2}^{3}-tc_{1,4}^{3} &
    F(2,4,1) &= c_{4,0}^{1}-c_{3,1}^{2} &&\\
    F(5,0,5) &= c_{0,3}^{4} &
    F(4,2,3) &= c_{2,2}^{3} - c_{0,3}^{4} -tc_{2,3}^{2} &
    F(4,1,5) &= c_{1,4}^{3}-tc_{1,5}^{2} && \!\!\diamond
  \end{xalignat*}
\end{example}

We next describe the linear relations among the equations $F(i,u,v)$.
By convention, we set $F(j,r,s) = 0$ if $r<0$, $s<0$ or $x^ry^s \in M$.

\begin{lemma}
  \label{l:relation}
  If $x^uy^v \not \in M$ with $u < p_{i-1}$ and $v < q_i - q_{i-1}$,
  then we have the relation
  \begin{multline}
    \label{eq:relation}
    0 = \sum_{j=i}^{\sigma} F(j,u - p_{i-1} + p_{j-1},v
    - q_i + q_j) \\
    + \sum_{\lambda \geq 0} \sum_{j = \sigma+1}^n t^{\lambda} F(j,u -
    p_{i-1}+p_{j-1} - \lambda \ell, v - q_i + q_j - \lambda m) \, .
  \end{multline}
\end{lemma}

\begin{proof}
  We first consider the summands with $j \leq \sigma$.  Since we have
  the inequalities $v < q_i$ and $v-q_i+q_j < q_j \leq q_\sigma \leq
  m+q_k$, the monomials $x^{u-p_{i-1}+p_{j-1}}y^{v-q_i+q_{j-1}}$ and
  $x^{u-p_{i-1}+p_j}y^{v-q_i+q_j}$ do not belong to $\langle
  x^{\ell+p_j}y^{m+q_j} : j \geq k \rangle$.  Hence, we have
  \[
  F(j,u-p_{i-1}+p_{j-1},v-q_i+q_j) =
  c_{u-p_{i-1}+p_{j-1},v-q_i+q_{j-1}}^{j-1} -
  c_{u-p_{i-1}+p_j,v-q_i+q_j}^j
  \]
  so the first part of the relation \eqref{eq:relation} telescopes to
  \[
  \sum_{j=i}^{\sigma} F(j,u - p_{i-1} + p_{j-1},v - q_i + q_j) = -
  c_{u-p_{i-1}+p_\sigma,v-q_i+q_\sigma}^\sigma \, ,
  \]
  because $c_{u,v-q_i+q_{i-1}}^{i-1} = 0$.  Thus, it remains to
  analyze the variables $c_{r,s}^j$ in the double sum
  \begin{equation}
    \label{eq:double}
    \sum_{\lambda \geq 0} \sum_{j = \sigma+1}^n t^{\lambda} F(j,u -
    p_{i-1}+p_{j-1} - \lambda \ell, v - q_i + q_j - \lambda m) \, .
  \end{equation}

  To begin, we consider $j = n$.  The only equations that might
  contain $c_{r,s}^{n}$ have the form
  $F(n,u-p_{i-1}+p_{n-1}-\lambda\ell, v-q_i+q_n-\lambda m)$; in this
  equation, such variables have the form
  $c_{u-p_{i-1}-(\mu+\lambda)\ell, v-q_i+q_n-(\mu+\lambda)m}^n$ for
  some $\mu \geq 0$ since $p_n=0$. Since $u < p_{i-1}$ and $\ell > 0$,
  it follows that $u-p_{i-1}-(\mu+\lambda)\ell < 0$.  Hence, no
  variable of the form $c_{r,s}^n$ appears in the double sum.

  Next, suppose that $\sigma < j < n$.  We first show that $c_{r,s}^j$
  appears in at most two equations of the form \eqref{eq:linear}.
  Specifically, if $x^ry^{s-q_j+q_{j+1}}$ reduces modulo
  $I_{\alpha}(t)$ to the standard monomial $x^{r+\mu
    \ell}y^{s-q_j+q_{j+1}+ \mu m}$ for some $\mu \in \NN$, then the
  variable $c_{r,s}^j$ appears in the equation $F(j+1,r+\mu \ell,
  s-q_j+q_{j+1}+\mu m)$ with coefficient $t^\mu$.  Otherwise
  $x^ry^{s-q_j+q_{j+1}}$ reduces to zero modulo $I_{\alpha}(t)$ and
  the variable $c_{r,s}^j$ does not appear in an equation of the form
  $F(j+1,r',s')$ for any $r', s' \in \NN$.  Similarly, if the monomial
  $x^{r+p_{j-1}-p_j}y^s$ reduces modulo $I_{\alpha}(t)$ to the
  standard monomial $x^{r+p_{j-1}-p_j+\mu' \ell}y^{s+ \mu' m}$ for
  some $\mu' \in \NN$, then $c_{r,s}^j$ appears in $F(j,r+p_{j-1}-p_j
  + \mu' \ell, s + \mu' m)$ with coefficient $-t^{\mu'}$.  Otherwise
  $x^{r+p_{j-1}-p_j}y^{s}$ reduces to zero modulo $I_{\alpha}(t)$ and
  the variable $c_{r,s}^j$ does not appear in an equation of the form
  $F(j,r',s')$ for any $r', s' \in \NN$.  In summary, the variable
  $c_{r,s}^j$ appears in at most two equations of the form
  \eqref{eq:linear} and when it appears the coefficient is uniquely
  determined.

  To complete this case, we show that if the variable $c_{r,s}^j$
  appears in the double sum \eqref{eq:double} then it appears twice:
  once with coefficient $t^{\nu}$ and once with coefficient
  $-t^{\nu}$.  The equation $F(j+1,r+\mu \ell, s-q_j+q_{j+1}+\mu m)$
  occurs in the double sum if and only if
  \begin{align*}
    \begin{bmatrix}
      r+\mu \ell \\ s-q_j+q_{j+1} + \mu m
    \end{bmatrix}
    &= 
    \begin{bmatrix}
      u - p_{i-1} +p_j - \lambda \ell \\ v - q_i + q_{j+1} - \lambda m
    \end{bmatrix} 
    \quad \text{for some $\lambda \in \NN$.}
  \end{align*}
  Similarly, $F(j,r+p_{j-1}-p_j + \mu' \ell, s + \mu' m)$ occurs if
  and only if
  \begin{align*}
    \begin{bmatrix}
      r+p_{j-1}-p_j + \mu' \ell \\  s+\mu' m 
    \end{bmatrix}
    &= 
    \begin{bmatrix}
      u - p_{i-1} +p_{j-1} - \lambda' \ell \\ v - q_i + q_j - \lambda' m
    \end{bmatrix} 
    \quad \text{for some $\lambda' \in \NN$.}
  \end{align*}
  Rearranging these equations, it follows that $c_{r,s}^j$ appears in
  double sum only if
  \begin{equation}
    \label{eq:key} 
    \begin{bmatrix} 
      r \\ s 
    \end{bmatrix}
    =
    \begin{bmatrix}
      u - p_{i-1} + p_j\\ v - q_i +q_j
    \end{bmatrix}
    -
    \nu 
    \begin{bmatrix}
      \ell \\ m 
    \end{bmatrix}
    \quad \text{for some $\nu \in \NN$;}
  \end{equation}
  either $\nu := \mu + \lambda$ and the coefficient of $c_{r,s}^j$ is
  $t^\nu$ or $\nu := \mu' + \lambda'$ and the coefficient of
  $c_{r,s}^j$ is $- t^\nu$.  On the other hand, if \eqref{eq:key} holds
  for some $\nu \in \NN$, then we have
  \begin{align*}
    \begin{bmatrix}
      r \\ s-q_j+q_{j+1}
    \end{bmatrix}
    +
    \nu 
    \begin{bmatrix}
      \ell \\ m 
    \end{bmatrix}
    &= 
    \begin{bmatrix}
      u-p_{i-1}+p_j \\ v-q_i+q_{j+1}
    \end{bmatrix} \quad \text{and} \\
    \begin{bmatrix}
      r+p_{j-1}-p_j \\ s
    \end{bmatrix}
    +
    \nu 
    \begin{bmatrix}
      \ell \\ m 
    \end{bmatrix}
    &= 
    \begin{bmatrix}
      u-p_{i-1}+p_{j-1} \\ v-q_i+q_j
    \end{bmatrix}  \, .
  \end{align*}
  Since $u < p_{i-1}$ and $v < q_i$, the monomials
  $x^{u-p_{i-1}+p_j}y^{v-q_i+q_{j+1}}$ and
  $x^{u-p_{i-1}+p_{j-1}}y^{v-q_i+q_j}$ do not belong to $M$, so both
  of the monomials $x^ry^{s-q_j+q_{j+1}}$ and $x^{r+p_{j-1}-p_j}y^s$
  reduce modulo $I_\alpha(t)$ to standard monomials of $M$.  Hence,
  the variable $c_{r,s}^j$ appears in twice in \eqref{eq:double} with
  $\lambda := \nu - \mu \geq 0$ and $\lambda' := \nu - \mu' \geq 0$.
  We conclude that, when $c_{r,s}^j$ appears in the double sum, it
  appears twice with the same coefficient in $t$ but with opposite
  signs.

  Lastly, assume that $j = \sigma$.  In this case, the variable
  $c_{r,s}^\sigma$ appears in at most three of the equations of the
  form \eqref{eq:linear}; it could appear in $F(\sigma+1,
  u-p_{i-1}+p_{\sigma}-\mu \ell, v-q_{i}+q_{\sigma+1} -\mu m)$ with
  coefficient $t^\mu$ for some $\mu \geq 0$, in $F(\sigma,
  u-p_{i-1}+p_{\sigma-1}-\mu' \ell, v-q_{i}+q_{\sigma} -\mu' m)$ with
  coefficient $-t^{\mu'}$ for some $\mu' \geq 0$, and in
  \[
  F \bigl(k,u-p_{i-1}+ p_{k-1}-p_{\sigma}-(\mu''+1)\ell, v-q_i+q_{k} -
  (\mu''+1) m \bigr)
  \] 
  with coefficient $-t^{\mu''+1}$ for some $\mu'' \geq 0$.  As in the
  previous case, $c_{r,s}^{\sigma}$ appears in \eqref{eq:double} if
  and only if \eqref{eq:key} holds for some $\nu \geq 0$.  However,
  only the first and third equation appear in the double sum, because
  the inner sum of \eqref{eq:double} starts at $j = \sigma+1$.  As a
  consequence, if \eqref{eq:key} holds with $\nu > 0$, then $c_{r,s}^j$
  appears in \eqref{eq:double} precisely twice with the same exponent
  on $t$ but with opposite signs.  Moreover, if \eqref{eq:key} holds
  with $\nu = 0$, then $c^j_{r,s}$ appears in \eqref{eq:double} only
  in the equation $F(\sigma+1, u-p_{i-1}+p_\sigma,
  v-q_i+q_{\sigma+1})$ with coefficient one.  In summary, we have
  established that
  \[
  c_{u-p_{i-1}+p_\sigma,v-q_i+q_\sigma}^\sigma = \sum_{\lambda \geq 0}
  \sum_{j = \sigma+1}^n t^{\lambda} F(j,u - p_{i-1}+p_{j-1} - \lambda
  \ell, v - q_i + q_j - \lambda m)
  \]
  as required.
\end{proof}

Using Lemma~\ref{l:relation}, we can describe the tangent space to
$\Hilb_S^h$ at the point corresponding to $I_{\alpha}(t)$ by a smaller
system of linear equations.

\begin{corollary} 
  \label{c:mingensJ} 
  If $M$ is a monomial ideal in $S = \kk[x,y]$ with Hilbert function
  $h \colon A \to \NN$ and $\alpha \in T_+(M)$, then the tangent space
  to $\Hilb_S^h$ at the point corresponding to $I_{\alpha}(t)$ is
  isomorphic to the subspace of $\AA^r$ cut out by
  \[
  \mathcal{G} := \left\{ F(i,u,v) : 
    \begin{array}[c]{p{220pt}}
      $(i,u,v)$ is an arrow for $M$ with $1 \leq i \leq n$
      and either $u \geq p_{i-1}$ or $v \geq q_i-q_{i-1}$
    \end{array}
  \right\} \, .
  \]  
\end{corollary}

\begin{proof}
  Since Proposition~\ref{p:linear} establishes that the tangent space
  is cut by all of the equations $F(i,u,v)$, it suffices to show that
  the $F(i,u,v)$ with $u < p_{i-1}$ and $v < q_i-q_{i-1}$ can be
  written as a linear combination of equations $F(i',u',v')$ not of
  this form.  We induct on $q_i-q_{i-1}-v$.  If $0 \geq
  q_i-q_{i-1}-v$, then the claim is vacuously true.  Otherwise,
  consider the expression for $F(i,u,v)$ given by
  Lemma~\ref{l:relation}.  For $j > i$, we have $v-q_i+q_j-\lambda m
  \geq q_j-q_{j-1}$, because $m < 0$ implies that $q_i-q_{j-1}+\lambda
  m \leq 0$.  Hence, the only terms in this expression that might not
  be in $\mathcal{G}$ have the form $F(i,u-\lambda \ell, v - \lambda
  m)$.  But these terms can be written as a linear combination of the
  elements of $\mathcal{G}$ by the induction hypothesis.
\end{proof}

\begin{example}
  \label{e:arrowsRel}
  If $A = 0$, $M = \langle x^4, x^2y, y^2 \rangle$, and $\alpha =
  (1,3,0) \in T_+(M)$ as in Example~\ref{e:arrowsL}, then
  Lemma~\ref{l:relation} applied to $i=1$ and $x, x^2, x^3 \not
  \in M$ shows $0 = F(1,1,0) + t F(1,0,1)$, $0 = F(1,2,0) + F(2,0,1) +
  t F(1,1,1)$, and $0 = F(1,3,0) + F(2,1,1)$.  \hfill $\diamond$
\end{example}

\begin{example}
  As in Example~\ref{e:L}, suppose that $A = 0$, $M = \langle x^7,
  x^6y, x^5y^2, x^4y^3, x^2y^4, y^6 \rangle$, and $\alpha = (4,3,2)
  \in T_+(M)$.  For $i=1$ and $x^5 \not \in M$, Lemma~\ref{l:relation}
  provides the relation
  \[
  0 = F(1,5,0) + F(2,4,1) + F(3,3,2) + F(4,2,3) + F(5,0,5) + t
  F(3,2,4) + t F(4,1,5) \, .
  \]
  Since $x^2y^4 \in M$, we have $F(3,2,4) = 0$ by convention. \hfill
  $\diamond$
\end{example}

The following theorem is the essential result in this section.  The
proof shows that the dimension of the appropriate linear subspace of
$\mathbb A^r$ equals the number of significant arrows.

\begin{theorem} 
  \label{t:tangentBasis}
  Let $M$ be a monomial ideal in $S = \kk[x,y]$ with Hilbert function
  $h \colon A \to \NN$ and fix $\alpha \in T_+(M)$.  The significant
  arrows $T(M)$ of $M$ index a basis for the tangent space to
  $\Hilb^h_S$ at the point corresponding to the edge ideal
  $I_\alpha(t)$ for all $t \in \kk$.
\end{theorem}

\begin{proof}
  Let $I := I_{\alpha}(t)$ and $\alpha = (k, \ell+p_k, m +q_k) \in
  T_+(M)$.  By Corollary~\ref{c:mingensJ}, we see that the tangent
  space to $\Hilb_S^h$ at the point corresponding to $I$ is isomorphic
  to the subspace of $\AA^r$ cut out by those $F(i,u,v)$ where $x^uy^v
  \not\in M$, $1 \leq i \leq n$, and either $u \geq p_{i-1}$ or $v
  \geq q_i-q_{i-1}$.  It suffices to show that the insignificant
  arrows are in bijection with the initial terms (or leading
  variables) in this system of equations.

  By definition, each $c_{u,v}^{i}$ corresponds to an arrow $(i,u,v)$
  associated to $M$.  For convenience, we say that the variable
  $c_{u,v}^{i}$ is significant, nonnegative, etc., whenever same
  adjective applies to the corresponding arrow.  Let $\order$ be a
  monomial order on the polynomial ring $\kk[c_{u,v}^{i} :
  \text{$(i,u,v)$ is an arrow of $M$}]$ satisfying the following
  conditions:
  \begin{align}
    \label{C0}
    \tag{\textsf{C0}} 
    & 
    \begin{array}[t]{p{410pt}} 
      if $r-p_j > u-p_i$ then $c^j_{r,s} \order c^i_{u,v}$;
    \end{array} \\
    \label{C1}
    \tag{\textsf{C1}} 
    &
    \begin{array}[t]{p{410pt}} 
      for two nonnegative variables $c_{u,v}^{i}$ and $c_{r,s}^{j}$
      with $r-p_j=u-p_i$, the inequality $i > j$ implies that
      $c_{u,v}^{i} \order c_{r,s}^{j}$;
    \end{array} \\
    \label{C2}
    \tag{\textsf{C2}} 
    &
    \begin{array}[t]{p{410pt}} 
      for two nonpositive variables $c_{u,v}^{i}$ and $c_{r,s}^{j}$
      with $r-p_j=u-p_i$, the inequality $i <j$ 
      implies that $c_{u,v}^{i} \order c_{r,s}^{j}$;
    \end{array} \\
    \label{C3}
    \tag{\textsf{C3}} 
    &
    \begin{array}[t]{p{410pt}} 
      for two utterly insignificant variables $c_{u,v}^{i}$ and
      $c_{r,s}^{j}$ with $r-p_j = u-p_i$, the inequality $i < j$
      implies that $c_{u,v}^{i} \order c_{r,s}^{j}$.
    \end{array}
  \end{align} 
  Observe that each equation $F(i,u,v)$ is homogeneous with respect
  to the grading defined by setting $\deg(c_{r,s}^j)$ equal to be the
  image of $\bigl[
    \begin{smallmatrix} 
      r-p_j \\ s-q_j 
    \end{smallmatrix} \bigr]$ in $\ZZ^2/ \ZZ
  \left[ 
    \begin{smallmatrix} 
      \ell \\ m 
    \end{smallmatrix} \right]$.  Hence, \eqref{C0} can be viewed as
  giving $t$ a negative weight.  Since $\ell > 0$, for each nonzero
  sum of the form $\sum\limits_{\mu \geq 0} t^\mu c_{u'-\mu \ell,
    v'-\mu m}^j$, condition~\eqref{C0} implies that $\initial_{\order}
  \Bigl( \textstyle\sum\limits_{\mu \geq 0} t^\mu c_{u'-\mu \ell,
    v'-\mu m}^j \Bigr) = t^{\widetilde{\mu}} c_{u'- \widetilde{\mu}
    \ell, v'- \widetilde{\mu} m}^j$, where $\widetilde{\mu} :=
  \min\{\mu : c_{u'- \mu \ell, v'-\mu m}^j \neq 0 \}$.  It follows
  that
  \begin{multline*}
    \initial_{\order} \bigl( F(i,u,v) \bigr) = \initial_{\order}
    \bigl( t^{\widetilde{\mu}'} c_{u- \widetilde{\mu}'
      \ell,v+q_{i-1}-q_i- \widetilde{\mu}' m}^{i-1} -
    t^{\widetilde{\mu}} c_{u-p_{i-1}+p_i- \widetilde{\mu} \ell,
      v- \widetilde{\mu} m}^{i} \\ - \delta_{i,k}
    t^{\widetilde{\mu}'' + 1} c_{u- p_{k-1} + p_\sigma -
      (\widetilde{\mu}'' +1)\ell,
      v-q_k+q_\sigma-(\widetilde{\mu}'' +1)m}^{\sigma} \bigr)
  \end{multline*}
  for appropriate $\widetilde{\mu}', \widetilde{\mu},
  \widetilde{\mu}'' \in \NN$; \eqref{C0} also guarantees that the
  initial term is the variable accompanying the smallest exponent of
  $t$.
  
  As a first step in constructing the bijection, we show that every
  insignificant arrow corresponds to the initial term of an element in
  $\mathcal{G}$.  We divide the analysis into three cases.

  \vspace{0.55em}
  \paragraph*{\textsc{nonnegative case}}  
  If $(i,u,v)$ is a nonnegative insignificant arrow, then
  $x^{u+p_{i-1}-p_i}y^v$ does not belong to $M$.  Since $u \geq p_i$,
  we have $u+p_{i-1}-p_i \geq p_{i-1}$, so the variable $c_{u,v}^{i}$
  appears as a nonzero term in $F(i,u+p_{i-1}-p_i,v) \in \mathcal{G}$
  (i.e.\  $\widetilde{\mu} = 0$).  Hence, \eqref{C0} and \eqref{C1}
  establish that $\initial_{\order} \bigl( F(i,u+p_{i-1}-p_i,v) \bigr)
  = c_{u,v}^{i}$.

  \vspace{0.55em}
  \paragraph*{\textsc{nonpositive case}}  
  If $(i,u,v)$ is a nonpositive insignificant arrow, then
  $x^uy^{v-q_i+q_{i+1}} \not\in M$.  Since $v\geq 0$, $v -q_i+q_{i+1}
  \geq -q_i+q_{i+1}$, the variable $c_{u,v}^{i}$ appears as a nonzero
  term in $F(i+1,u,v-q_i+q_{i+1}) \in \mathcal{G}$ (so
  $\widetilde{\mu}' = 0$).  Together \eqref{C0} and \eqref{C2}
  establish that $\initial_{\order} \bigl( F(i+1,u,v-q_i+q_{i+1})
  \bigr) = c_{u,v}^{i}$.

  \vspace{0.55em}
  \paragraph*{\textsc{utterly insignificant case}} 
  If $(i,u,v)$ is an utterly insignificant arrow, then we have $u <
  p_i \leq p_j$ for $j \leq i$ and $v < v-q_i+q_{i+1} < q_{i+1} \leq
  q_j$ for $j > i$, so $x^uy^{v-q_i+q_{i+1}} \not\in M$.  Hence, the
  variable $c_{u,v}^{i}$ appears as a nonzero term in
  $F(i+1,u,v-q_i+q_{i+1}) \in \mathcal{G}$.  Hence, \eqref{C0} and
  \eqref{C3} establish that $\initial_{\order} \bigl(
  F(i+1,u,v-q_i+q_{i+1}) \bigr) = c_{u,v}^{i}$.

  \vspace*{0.55em}
  \noindent
  By combining these three cases, we get an injective map from the
  insignificant arrows of $M$ to the elements of $\mathcal{G}$.

  To establish that this map is a bijection, we show that the initial
  term of each element of $\mathcal{G}$ corresponds to an
  insignificant arrow.  Again, there are three cases.  Fix $x^uy^v
  \not \in M$.

  \vspace{0.55em}
  \paragraph*{\textsc{nonnegative case}}    
  If $u \geq p_{i-1}$, then we have the inequalities $u >
  u-p_{i-1}+p_i \geq p_i$.  Hence, $(i,u-p_{i-1}+p_i,v)$ is an
  insignificant nonnegative arrow for $M$.  Moreover, \eqref{C0} and
  \eqref{C1} ensure that $\initial_{\order} \bigl( F(i,u,v) \bigr) =
  c_{u-p_{i-1}+p_i,v}^{i}$.

  \vspace{0.55em}
  \paragraph*{\textsc{nonpositive case}}  
  If $v \geq q_i$, then we have the inequalities $v > v+q_{i-1}-q_i
  \geq q_{i-1}$.  Hence, $(i-1,u,v+q_{i-1}-q_i)$ is an insignificant
  nonnegative arrow for $M$.  The inequality $u > u-p_{i-1}+p_i$
  together with \eqref{C0} and \eqref{C2} imply that
  $\initial_{\order} \bigl( F(i,u,v) \bigr) =
  c_{u,v+q_{i-1}-q_i}^{i-1}$.

  \vspace{0.55em}
  \paragraph*{\textsc{utterly insignificant case}} 
  If $u < p_{i-1}$ and $q_i - q_{i-1} \leq v < q_i$, then we have the
  inequalities $\min(v,q_{i-1}) > v+q_{i-1}-q_i \geq 0$.  Hence,
  $(i-1,u,v+q_{i-1}-q_i)$ is an utterly insignificant arrow for $M$.
  The inequality $u > u-p_{i-1}+p_i$ together with \eqref{C0} and
  \eqref{C3} imply that $\initial_{\order} \bigl( F(i,u,v) \bigr) =
  c_{u,v+q_{i-1}-q_i}^{i-1}$.

  \vspace*{0.55em}
  \noindent
  In each case, the initial term of $F(i,u,v)$ is an insignificant
  arrow.  Moreover, if we have $(i,u,v) \neq (j,r,s)$ with $F(i,u,v),
  F(j,r,s) \in \mathcal{G}$, then $F(i,u,v)$ and $F(j,r,s)$ have
  different initial terms.  Therefore, we have a bijection between the
  insignificant arrows of $M$ and the initial terms of the elements of
  $\mathcal{G}$.

  Since the initial terms for the elements of $\mathcal{G}$ are
  relatively prime, they form a Gr\"obner basis with respect to
  $\order$.  Therefore, $T(M)$ indexes a basis for the tangent space
  to $\Hilb_S^h$ at the point corresponding to $I_{\alpha}(t)$ for all
  $t \in \kk$.
\end{proof}

\begin{example}
  If $A = 0$, $M = \langle x^4, x^2y, y^2 \rangle$, and $\alpha =
  (1,3,0) \in T_+(M)$ as in Example~\ref{e:arrowsRel}, then
  Theorem~\ref{t:tangentBasis} shows that the tangent space to the
  appropriate multigraded Hilbert scheme at the point corresponding to
  $I_{\alpha}(t)$ is isomorphic to subspace cut out by
  \begin{align*}
    \langle \mathcal{G} \rangle &:= \langle F(1,0,1), F(1,1,1),
    F(2,2,0), F(2,3,0), F(2,0,1), F(2,1,1) \rangle \\
    &= \langle c_{0,0}^{0}, c_{1,0}^{0}, -c_{0,0}^{2},
    -c_{1,0}^{2}-tc_{0,1}^{2}+tc_{2,0}^{1}+t^2c_{1,1}^{1},
    c_{0,0}^{1}, c_{1,0}^{1} + t c_{0,1}^{1} \rangle \, .
  \end{align*}
  Hence, the tangent space has dimension $(3)(6) - 6 = 12$.\hfill
  $\diamond$
\end{example}

\section{Smoothness}
\label{s:smoothness}

The goal of this final section is to prove Theorem~\ref{t:main}.  To
begin, we show that $\Hilb_S^h$ has at least one nonsingular point.
This result parallels \cite{ReevesStillman}*{Theorem~1.4} and our
proof extends the techniques in \cite{Evain}*{Proposition~10}.

\begin{proposition} 
  \label{p:lexmostsmooth} 
  Let $S = \kk[x,y]$ and let $L_h$ be the lex-most ideal for a Hilbert
  function $h \colon A \to \NN$ satisfying $|h| < \infty$.  The ideal
  $L_h$ corresponds to a nonsingular point on $\Hilb_S^h$.
\end{proposition}

\begin{proof}
  Let $d$ be the number of significant arrows associated to $L_h$.  By
  Theorem~\ref{t:tangentBasis}, $d$ equals the dimension of the
  tangent space to $\Hilb_S^h$ at the point corresponding to $L_h$.
  Thus, it suffices to show that the dimension of $\Hilb_S^h$ at this
  point is at least $d$.  We accomplish this by constructing a map
  $\tau \colon \AA^d \to \Hilb_S^h$ in which $\tau(0)$ corresponds to
  $L_h$ and the dimension of the image is $d$.  Since $\Hilb_S^h$ is a
  fine moduli space, the map $\tau \colon \AA^d \to \Hilb_S^h$ is
  determined by an admissible ideal $I$ in $R := K[x,y]$ for $K:= \kk
  \bigl[ c_{u,v}^i : (i,u,v) \in T(M) \bigr]$.  We may regard $I$ as a
  family of ideals over the base $\AA^d = \Spec(K)$.

  We define the generators of $I$ recursively.  Since
  Proposition~\ref{p:lexmost} states $T_+(L_h) = \varnothing$, the
  significant arrows associated to $L_h$ are either nonpositive or
  have the form $(i,p_i,v)$.  For $1 \leq i \leq n$, consider $g_i :=
  y^{-q_{i-1}+q_i} + \textstyle\sum\nolimits_{(i,p_i,v) \in T(L_h)}
  c_{p_i,v}^i \, y^{v-q_{i-1}}$.  Since $(i,p_i,v) \in T(L_h)$, we
  have $x^{p_{i-1}}y^v \in M$, so $v \geq q_{i-1}$ and $g_i$ is a
  polynomial in $R = K[x,y]$.  Setting $f_n := \prod_{i=1}^n g_i$
  means $\initial_{\lex}(f_n) = \prod_{i=1}^n y^{-q_{i-1}+q_i} =
  y^{q_n} = x^{p_n}y^{q_n}$ because $p_n = q_0 = 0$.  Next, suppose
  that the polynomials $f_{i+1},\dots, f_n$ are defined, with
  $\prod_{k=1}^j g_k$ dividing $f_j$ for $i+1 \leq j \leq n$.  Given
  $(i,u,v) \in T_{\leq 0}(L_h)$, we have $x^uy^{v-q_i+q_{i+1}} \in
  L_h$, so the minimal monomial generator $x^{p_j}y^{q_j}$ divides
  $x^uy^{v-q_i+q_{i+1}}$ for some index $j$ such that $i < j \leq n$.
  Let $\varepsilon =\varepsilon(i,u,v) := \max\{ j :
  \text{$x^{p_j}y^{q_j}$ divides $x^uy^{v-q_i+q_{i+1}}$} \}$ and, for
  $0 \leq i < n$, define
  \[
  f_i := \frac{1}{g_{i+1}} \biggl( x^{p_i-p_{i+1}} f_{i+1} +
  \sum\limits_{(i,u,v) \in T_{\leq 0}(L_h)} c_{u,v}^i \,
  x^{u-p_{\varepsilon}}y^{v-q_i+q_{i+1}-q_{\varepsilon}}
  f_{\varepsilon} \biggr) \, .
  \]
  Since $\prod_{k=1}^{i+1} g_{k}$ divides $f_j$ for all $j>i$, it
  follows that $f_i \in R$ with $\prod_{k=1}^i g_k$ dividing $f_i$.
  Repeating this process, we can define $f_i \in R$ for $0 \leq i \leq
  n$.  Moreover, the equation $y^{-q_i+q_{i+1}} \initial_{\lex}(f_i) =
  \initial_{\lex}(g_{i+1}f_i) = x^{p_i-p_{i+1}}
  \initial_{\lex}(f_{i+1})$ establishes that $\initial_{\lex}(f_i) =
  x^{p_i}y^{q_i}$.  With this notation, we define the ideal $I :=
  \langle f_0, \dotsc, f_n \rangle \subseteq K[x,y]$.

  We next show that $\initial_{\lex} \bigl(I \otimes_K k(\mathfrak{p})
  \bigr) = L_h \otimes_K k(\mathfrak{p})$ where $k(\mathfrak{p}) :=
  K_{\mathfrak{p}}/\mathfrak{p} K_{\mathfrak{p}}$ is the residue field
  of the point $\mathfrak{p} \in \Spec(K)$.  Since the minimal
  generators of $L_h$ are the initial terms with respect to $\lex$ of
  the defining generators of $I(\mathfrak{p}) := I \otimes_K
  k(\mathfrak{p})$, it suffices to show that these generators form a
  Gr\"{o}bner basis.  By Buchberger's criterion
  \cite{Eisenbud}*{Exercise~15.19}, we need only prove that the
  S-polynomials reduce to zero modulo the generators of
  $I(\mathfrak{p})$ for pairs of generators corresponding to the
  minimal syzygies of $L_h$.  The minimal syzygies of a monomial ideal
  in $R(\mathfrak{p}) := R \otimes_K k(\mathfrak{p}) =
  k(\mathfrak{p})[x,y]$ are indexed by adjacent pairs of minimal
  generators; see \cite{MillerSturmfels}*{Proposition~3.1}.  The
  S-polynomial for the adjacent generators $f_{i-1},f_i$ is
  \begin{multline*}
    y^{-q_{i-1}+q_i}f_{i-1} - x^{p_{i-1}-p_i} f_i =
    (y^{-q_{i-1}+q_i}-g_i)f_{i-1} + \textstyle\sum\limits_{(i-1,u,v)
      \in T_{\leq 0}(L_h)} c_{u,v}^{i-1} \,
    x^{u-p_{\varepsilon}}y^{v-q_{i-1}+q_i-q_{\varepsilon}}
    f_{\varepsilon} \\
    =  \textstyle\sum\limits_{(i-1,u,v) \in T_{\leq 0}(L_h)}
    c_{u,v}^{i-1}
    x^{u-p_{\varepsilon}}y^{v-q_{i-1}+q_i-q_{\varepsilon}}
    f_{\varepsilon} \!-\! \Bigl( \textstyle\sum\limits_{(i,p_i,v) \in
      T(L_h)} \!\! c_{p_i,v}^i y^{v-q_{i-1}} \Bigr) f_{i-1} \, .
  \end{multline*}
  Since the initial terms of all the summands in the last expression
  are less than the monomial $y^{-q_{i-1}+q_i}
  \initial_{\lex}(f_{i-1})$, we conclude that this S-polynomial
  reduces to zero modulo the generators of $I(\mathfrak{p})$.  It
  follows from \cite{Eisenbud}*{Theorem~15.3} that $h \colon A \to
  \NN$ is the Hilbert function of $R(\mathfrak{p})/I(\mathfrak{p})$.

  We now use this to show that $I$ is admissible.  Since we have
  $\dim_{k(\mathfrak{p})} \bigl( R(\mathfrak{p})/I(\mathfrak{p})
  \bigr)_a = h(a)$ for all $\mathfrak{p} \in \Spec(K)$, Nakayama's
  Lemma implies that the $K_{\mathfrak{p}}$-module
  $(R_\mathfrak{p}/I_{\mathfrak{p}})_a$ requires at most $h(a)$
  generators.  However, the rank of the $K_{\mathfrak{p}}$-module
  $(R_\mathfrak{p}/I_{\mathfrak{p}})_a$ is also bounded above by 
  $\dim_{k(0)} \bigl( R(0)/I(0) \bigr)_a$ and the Hilbert function at
  the generic point $\langle 0 \rangle \in \Spec(K)$ also equals
  $h(a)$.  Hence, $(R/I)_a$ is a locally free $K$-module of constant
  rank $h(a)$ on $\Spec(K)$.  The map $\tau \colon \AA^d \to
  \Hilb_S^h$ determined by the admissible $R$-ideal $I$ is injective,
  so the dimension of the image is $d$.
\end{proof}

We conclude with the proof of the main result.

\begin{proof}[Proof of Theorem~\ref{t:main}]
  We first show that $\Hilb_S^h$ is nonsingular when $S = \kk[x,y]$
  and $\kk$ is a field.  An ideal in $S = \kk[x,y]$ with codimension
  greater than one has finite colength.  Hence we may assume, by
  Theorem~\ref{t:factoring}, that $|h| := \sum_{a \in A} h(a) <
  \infty$.  Given a closed point on $\Hilb_S^h$, Lemma~\ref{l:curve}
  shows that the Gr\"{o}bner degenerations of the corresponding ideal
  $I'$ give an irreducible rational curve on $\Hilb_S^h$ that contains
  the points corresponding to $I'$ and $\initial_{\xel}(I')$.  Since
  the dimension of the tangent space is upper semicontinuous, it
  suffices to demonstrate that each point on $\Hilb_S^h$ corresponding
  to a monomial ideal is nonsingular.  Theorem~\ref{t:connected}
  establishes that the points on $\Hilb_S^h$ corresponding to monomial
  ideals lie on a curve $C$ in which the irreducible components are
  associated to positive significant arrows.  It follows from
  Theorem~\ref{t:tangentBasis} that the dimension of the tangent space
  is weakly increasing as we move along $C$ from a point corresponding
  to a monomial ideal to the point corresponding to $L_h$.
  Proposition~\ref{p:lexmostsmooth} proves that the point
  corresponding to $L_h$ is nonsingular.  We conclude that dimension
  of the tangent space is constant along $C$ and $\Hilb_S^h$ is
  nonsingular.  Theorem~\ref{t:connected} also establishes that
  $\Hilb_S^h$ is connected, so it follows that $\Hilb_S^h$ is
  irreducible.

  To complete the proof, let $S = \ZZ[x,y]$ and let $\eta \colon
  \Hilb_S^h \to \Spec(\ZZ)$ be the canonical map.  To show $\eta$ is
  smooth, it suffices by \cite{EGA}*{Theorem~17.5.1} to demonstrate
  that $\eta$ is flat and, for each $\mathfrak{p} \in \Spec(\ZZ)$,
  that the fiber $\eta^{-1}(\mathfrak{p})$ is smooth over the perfect
  field $\ZZ_{\mathfrak{p}}/\mathfrak{p} \ZZ_{\mathfrak{p}}$.  Since
  each fiber $\eta^{-1}(\mathfrak{p})$ is
  $\Hilb_{\ZZ/\mathfrak{p}[x,y]}^{h}$ (for example, see
  \cite{HaimanSturmfels}*{Lemma~3.14}), combining the first paragraph
  with \cite{EGA}*{Corollaire~17.15.2} shows that each fiber is
  smooth.  The lex-most ideal on each fiber produces a section of
  $\eta$, which implies that $\eta$ is surjective.  The image of this
  section is irreducible, since $\Spec(\ZZ)$ is, and the first
  paragraph also shows the fibers $\eta^{-1}(\mathfrak{p})$ are all
  irreducible as well, so it follows that $\Hilb_S^h$ is irreducible.
  Hence, the underlying reduced scheme $(\Hilb_S^h)_{\text{red}}$ is
  irreducible and dominates $\Spec(\ZZ)$, so
  \cite{Hartshorne}*{Proposition~III.9.7} establishes that the
  canonical map $(\Hilb_S^h)_{\text{red}} \to \Spec(\ZZ)$ is flat.
  The fact that the nilradical of $\Hilb_S^h$ is the zero sheaf, so
  $\Hilb_S^h = (\Hilb_S^h)_{\text{red}}$, can then be deduced from the
  fact that each fiber $\eta^{-1}(\mathfrak{p})$ is reduced.
  Therefore, we conclude that $\Hilb_S^h$ is smooth and irreducible
  over $\ZZ$.
\end{proof}

\raggedright
 
\end{document}